\documentclass[11pt, a4]{amsart}

\usepackage{pdfpages}

\usepackage{enumerate}
\usepackage{amsthm,amsfonts,amsmath,amscd,amssymb}
\usepackage{xcolor}
\usepackage{tikz} 
\usepackage{graphicx}
\usepackage{xcolor}
\usepackage{caption}
\usepackage{subcaption}
\let\oldmarginpar\marginpar
\renewcommand\marginpar[1]{\-\oldmarginpar[\raggedleft\footnotesize #1]%
{\raggedright\footnotesize #1}}

\usepackage[backref, colorlinks=true, linkcolor=blue, citecolor=blue,urlcolor=blue,citebordercolor={0 0 1},urlbordercolor={0 0 1},linkbordercolor={0 0 1}]{hyperref}

\usepackage[margin=3cm]{geometry}

\theoremstyle{plain}
\newtheorem{theorem}{Theorem}[section]
\newtheorem{lemma}[theorem]{Lemma}

\newtheorem{proposition}[theorem]{Proposition}
\newtheorem{corollary}[theorem]{Corollary}

\theoremstyle{definition}
\newtheorem{definition}[theorem]{Definition}
\newtheorem{remark}[theorem]{Remark}

\theoremstyle{remark}

\numberwithin{equation}{section}

\newcommand{\Diff}{\operatorname{Diff}}
\newcommand{\Symp}{\operatorname{Symp}}
\newcommand{\Ham}{\operatorname{Ham}}
\newcommand{\Homeo}{\operatorname{Homeo}}

\newcommand{\id}{\text{id}}
\newcommand{\Tw}{\text{Tw}}
\newcommand{\anti}{\text{anti}}
\newcommand{\inv}{\text{inv}}

\newcommand{\calE}{\mathcal{E}}

\newcommand{\calS}{\mathcal{S}}

\newcommand{\cP}{\mathcal{P}}

\newcommand{\scrX}{\EuScript{X}}

\newcommand{\scrC}{\EuScript{C}}

\newcommand{\scrD}{\EuScript{D}}

\newcommand{\scrW}{\EuScript{W}}
\newcommand{\scrF}{\EuScript{F}}
\newcommand{\cW}{\EuScript{W}}

\newcommand{\scrL}{\EuScript{L}}

\newcommand{\calO}{\mathcal{O}}

\newcommand{\bF}{\mathbb{F}}
\newcommand{\bR}{\mathbb{R}}
\newcommand{\bZ}{\mathbb{Z}}

\newcommand{\bC}{\mathbb{C}}

\newcommand{\bP}{\mathbb{P}}

\newcommand{\Coh}{\operatorname{Coh}}
\newcommand{\Ext}{\operatorname{Ext}}
\newcommand{\Tor}{\operatorname{Tor}}
\newcommand{\Pic}{\operatorname{Pic}}
\newcommand{\Stab}{\operatorname{Stab}}
\newcommand{\Spec}{\operatorname{Spec}}
\newcommand{\Supp}{\operatorname{Supp}}
\newcommand{\Ann}{\operatorname{Ann}}
\newcommand{\Conf}{\operatorname{Conf}}
\newcommand{\Hom}{\operatorname{Hom}}
\newcommand{\supp}{\operatorname{supp}}
\newcommand{\Fix}{\operatorname{Fix}}
\newcommand{\Lag}{\operatorname{Lag}}
\newcommand{\Int}{\operatorname{Int}}

\newcommand{\cO}{\mathcal{O}}
\newcommand{\Auteq}{\operatorname{Auteq}}
\newcommand{\Aut}{\operatorname{Aut}}

\newcommand{\Ob}{\operatorname{Ob}}
\newcommand{\HF}{\operatorname{HF}}
\newcommand{\CF}{\operatorname{CF}}
\newcommand{\PBr}{\operatorname{PBr}}
\newcommand{\ch}{\operatorname{ch}}

\newcommand{\cdbar}{\overline{\partial}}

\newcommand{\ex}{\operatorname{ex}}
\newcommand{\pr}{\operatorname{pr}}
\newcommand{\gr}{\operatorname{gr}}
\newcommand{\pt}{\operatorname{pt}}

\renewcommand{\Re}{\operatorname{Re}}
\renewcommand{\Im}{\operatorname{Im}}

\usepackage{amscd,euscript}
\usepackage[frame,cmtip,curve,arrow,matrix,line,graph]{xy}

\usepackage{xcolor}

\begin{document}
\begin{abstract}
Let $X$ denote the `conifold smoothing', the symplectic Weinstein manifold which is the complement of a smooth conic in $T^*S^3$, or equivalently the plumbing of two copies of $T^*S^3$ along a Hopf link. Let $Y$ denote the `conifold resolution', by which we mean the complement of a smooth divisor in $\mathcal{O}(-1) \oplus \mathcal{O}(-1) \to \bP^1$.  We prove that the compactly supported symplectic mapping class group of $X$ splits off a copy of an infinite rank free group, in particular is infinitely generated; and we classify spherical objects in the bounded derived category $D(Y)$ (the three-dimensional `affine $A_1$-case'). Our results build on work of Chan-Pomerleano-Ueda and Toda, and both theorems make essential use of working on the `other side' of the mirror. 
\end{abstract}

\title{Symplectomorphisms and spherical objects in the conifold smoothing}


\author{Ailsa Keating} 
\address{Centre for Mathematical Sciences, University of Cambridge, Wilberforce Road, CB3 0WB, United Kingdom}
\email{amk50@cam.ac.uk}

\author{Ivan Smith} 
\address{Centre for Mathematical Sciences, University of Cambridge, Wilberforce Road, CB3 0WB, United Kingdom}
\email{is200@cam.ac.uk}

\maketitle

\tableofcontents

\section{Introduction}\label{sec:intro}

\subsection{Results}

Consider the open symplectic manifold
$$
X = \{ u_1 v_1 = z-1, \, u_2 v_2 = z+1 \} \subset \bC^4 \times \bC^\ast
$$
where $u_i, v_i \in \bC$ and $z \in \bC^\ast$, with the restriction $\omega$ of the standard exact K\"ahler form from $\bC^4 \times \bC^\ast$.  Projection to $z \in \bC^\ast$ defines a Morse-Bott-Lefschetz fibration on $X$, from which one sees that $X$ is also a plumbing of two copies of $T^*S^3$ along a Hopf link.  

\begin{theorem} \label{thm:main}
There is a split injection $\bZ^{\ast \infty} \to \pi_0\Symp_{c}(X,\omega)$, where $\bZ^{\ast \infty}$ denotes the free group on countably infinitely many generators. In particular, the symplectic mapping class group $\pi_0\Symp_{c}(X,\omega)$ is infinitely generated.
\end{theorem}

The generators of the image $\bZ^{\ast \infty}$ in $\pi_0\Symp_{c}(X,\omega)$ will be Dehn twists in certain Lagrangian spheres in $X$. 

\begin{corollary}
There is a 3-dimensional Stein domain $(X',\partial X')$, with both $X'$ and $\partial X'$ simply connected,  and with $\pi_0\Symp_c(X')$ infinitely generated.
\end{corollary}

The domain $X'$ is obtained by adding one subcritical Weinstein two-handle to $X$.
By contrast, for a simply-connected compact six-manifold $\widetilde{X}$ with simply connected boundary, $\pi_0\Diff_c(\widetilde{X})$ is always of finite type, cf.  \cite{Kupers} and Lemma \ref{lem:finite_rank}.  

The Milnor fibres of  the $A_k$-singularities for $k>1$ contain infinitely many exact Lagrangian spheres up to Lagrangian isotopy, but these are obtained from a single sphere under the action of the symplectic mapping class group. We show, in Corollary \ref{cor:infinite-sphere-orbits}, that the action of $\pi_0\Symp_c(X')$ on the set of isotopy classes of Lagrangian $3$-spheres in  $X'$ has infinitely many orbits, answering a folk question sometimes attributed to Fukaya.

The mirror $Y$ to $X$ is obtained from the small resolution $\mathcal{O}(-1) \oplus \mathcal{O}(-1) \to \bP^1$ of the 3-fold ordinary double point $xy=(1+w)(1+t)\subset\bC^4$ by removing the pullback of the divisor $\{wt=0\}$.  Our methods also resolve the long-standing open problem of classifying all the spherical objects in $D(Y)$ (i.e. classifying sphericals in the `three-dimensional affine $A_1$-case'). One consequence of this classification is:

\begin{theorem} \label{thm:main2}
The spherical objects in $D(Y)$ form a single orbit under a natural action of the pure braid group $\PBr_3$.
\end{theorem}

It may be worth highlighting that Theorem \ref{thm:main}, which is a purely symplectic statement, relies essentially on passing to the other side of mirror symmetry, to exploit a computation of a space of stability conditions accessible only via algebraic geometry; whilst Theorem \ref{thm:main2}, which is a purely algebro-geometric statement, relies essentially on passing back to the symplectic setting, to exploit constraints arising from Nielsen-Thurston theory and symplectic dynamics on surfaces.
Also, the dynamical analysis of Propositions \ref{prop:generalised_Khovanov_Seidel} and \ref{prop:growth} is closely related to ideas around categorical entropy. See Remark \ref{rmk:entropy} for a brief discussion.

\subsection{Context}

Understanding the symplectic mapping class group is a classical and long-standing problem in symplectic topology.  Theorem \ref{thm:main} is the first example of a (finite type) symplectic manifold for which the compactly supported symplectic mapping class group is known to be infinitely generated. It complements a result of \cite{Sheridan-Smith}, see also \cite{Smirnov}, which constructed compact symplectic manifolds for which the `symplectic Torelli group' (of mapping classes acting trivially on cohomology) was infinitely generated, and a result of \cite{Auroux-Smith} which gave a compact smooth manifold with a family of symplectic forms $\omega_t$ for which the rank of $\pi_0\Symp(X,\omega_t)$ grew unbounded with $t$. In contrast, we expect the symplectic mapping class group of a $K3$ surface, or of (the mirror to) a log Calabi-Yau surface, to be finitely generated.

The classification of spherical objects in either the derived category of an algebraic variety or the Fukaya category of a symplectic manifold is also a well-known and active area of research.  In particular, for the Milnor fibres of type $A$ surface singularities a large body of work yielded a complete classification \cite{KS,  IU, WeiweiWu}.  Much less is known on threefolds.  For `small' categories associated to flopping contractions, namely the full subcategory of objects supported on the exceptional locus and with trivial push-forward, a classification was recently obtained in \cite{HaraW}, see also \cite{SW}.  Theorem \ref{thm:main2} is the first example which treats the `big' category of all complexes supported on the exceptional locus, where one encounters an affine hyperplane arrangement for the associated chamber of flops (or stability conditions) rather than a finite hyperplane arrangement.

\subsection{Remarks on the proofs}
 Homological mirror symmetry for the conifold has been established by Chan, Pomerleano and Ueda in \cite{CPU}, both for the wrapped Fukaya category $\cW(X)$ and the compact subcategory generated by the `compact core' of $X$.  In particular, there is a $\bC$-linear equivalence $\cW(X) \simeq D(Y)$. Crucially, both categories are also linear over a larger ring $R = SH^0(X) \simeq \Gamma(\mathcal{O}_Y)$. The group of stability conditions on $D(Y)$, and the group of autoequivalences of that category,  was analysed by Toda \cite{Toda1, Toda2} (see also \cite{HW} for closely related results in a more general setting).  

The heart of the proof of Theorem \ref{thm:main} is to show that the part of the autoequivalence group of $\cW(X)$ coming from compactly supported symplectomorphisms surjects to the  subgroup of $R$-linear autoequivalences of $D(Y)$ which moreover act trivially on $K$-theory.  This uses in an essential way the detailed knowledge of $\Auteq D(Y)$ arising from the computation of the space of stability conditions.  By contrast, the key ingredient in the proof of Theorem \ref{thm:main2} is a growth rate estimate for ranks of Floer cohomology under autoequivalences, which is proved using a combination of localisation-type theorems (to reduce Floer theory computations to a two-dimensional surface inside the six-dimensional $X$) together with classical results of Nielsen and Thurston on dynamics of pseudo-Anosov maps.  

The symplectic topology of two-dimensional log Calabi-Yau surfaces has been fruitfully illuminated through mirror symmetry \cite{GHK, HK,HK2}. This paper  treats one simple  three-dimensional example, but many of the techniques and results seem likely to generalise. 

\subsection*{Acknowledgements} The authors are grateful to Spencer Dowdall, Paul Hacking, Daniel Pomerleano, Oscar Randal--Williams, Paul Seidel and Michael Wemyss for helpful conversations and correspondence. We thank the anonymous referee for their careful reading and detailed feedback on the first version of this paper.

This material is based upon work supported by the National Science Foundation under Grant No.~1440140, while the authors were in residence at the Mathematical Sciences Research Institute in Berkeley, California, during the Fall semester 2022.   A.K.~was partially supported by EPSRC Open Fellowship EP/W001780/1 and ERC Starting Grant 101041249. I.S.~is grateful to the Clay Foundation for support at MSRI as a Clay Senior Scholar, and to EPSRC  for support through Frontier Research grant EP/X030660/1 (in lieu of an ERC Advanced Grant).

\section{Wrapped Fukaya categories of Liouville manifolds}

\subsection{Preliminary wrapped Fukaya categories}

Let $(W,\theta)$ denote a finite type complete Liouville manifold (i.e.~the completion of a Liouville domain $\mathring{W}$ by cylindrical ends). A Liouville domain $(\mathring{W},\theta)$ canonically determines a finite type complete Liouville manifold $(W,\theta)$. A Lagrangian is cylindrical in $(W,\theta)$ , resp.~$(\mathring{W},\theta)$, if it is invariant under the Liouville vector field outside a compact set, resp.~in a collar neighbourhood of the boundary. Note that any exact cylindrical Lagrangian for $(W, \theta)$ is exact Lagrangian isotopic, via the inverse Liouville flow, to a Lagrangian which is the completion of a cylindrical Lagrangian in $(\mathring{W}, \theta)$.

\begin{definition}\label{def:preliminary-wrapped}

The preliminary wrapped Fukaya  category $\cW_{\pr} (W, \theta)$ is defined as the wrapped Fukaya category set-up as in \cite[Section 2.4]{GPS2}, with the following choices: the category is ungraded, with coefficient ring $\bZ/2$; and Lagrangians are exact and cylindrical (no further brane data is required as there are no gradings). We will write $\cW_{\pr}(\mathring{W},\theta)$ for $\cW_{\pr}(W,\theta)$ when we wish to emphasise the choice of domain.

\end{definition}

`Preliminary' is intended as in \cite{Seidel-FCPLT}. We will later restrict to Lagrangians which admit brane structures including a choice of $Spin$-structure and grading, yielding a category $\cW(W, \theta)$ which is defined over $\bZ$ and $\bZ$-graded.

\subsection{Actions of  symplectomorphisms}\label{sec:symplecto-actions}

An exact symplectomorphism of $(\mathring{W}, \theta)$ is a diffeomorphism $\phi: \mathring{W} \to \mathring{W}$ such that $\phi^\ast \theta = \theta + df$, where $f$ needn't have support in $\Int(\mathring{W})$ (and is formally defined on a small open neighbourhood of $\mathring{W}$ in $W$). 
Let $\Symp (\mathring{W})$  be the group of all symplectomorphisms of $\mathring{W}$,  equipped with the $C^\infty$ topology. 
Let $\Symp_{\ex} (\mathring{W}, \theta)$ denote the subgroup of exact symplectomorphisms; the connected component of the identity in $\Symp_{\ex} (\mathring{W}, \theta)$ is $\Ham (\mathring{W}, \theta)$ the subgroup of Hamiltonian diffeomorphisms of $(\mathring{W}, \theta)$.

\begin{lemma}\label{lem:exact-action}
The group $ \pi_0 \Symp_{\ex} (\mathring{W}, \theta)$ acts on $\cW_{\pr}(W, \theta)$. 

\end{lemma}

\begin{proof}

Fix $\phi \in \Symp_{\ex} (\mathring{W}, \theta) $.  
Cylindrical Lagrangians for $(\mathring{W}, \theta)$ get mapped  under $\phi$ to cylindrical Lagrangians for $(\mathring{W}, \phi^\ast \theta)$. 
This induces a map $\phi: \cW_{\pr}(\mathring{W}, \theta) \to \cW_{\pr}(\mathring{W}, \phi^\ast \theta)$ taking a Lagrangian $L$ to $\phi (L)$ and pulling back all Floer data. 

On the other hand, as $\phi$ is an exact symplectomorphism of $\mathring{W}$, $\phi^\ast \theta = \theta + df$, for some smooth function $f$. We claim that $(\mathring{W}, \theta+ tdf)$ is a  Liouville domain for all $t \in [0,1]$. 
For $t=1$, this is true tautologically via pullback, with the Liouville flow $Z_{\theta+df}$ for $(\mathring{W}, \theta + df)$ being the pushforward under $\phi$ of the Liouville flow $Z_\theta$ for $(\mathring{W}, \theta)$.
Now notice, first, that 
$(\partial \mathring{W}, \ker (\theta + t df))$ is a contact manifold for all $t \in [0,1]$: the positivity one needs comes from interpolating between two volume forms of the same sign.
And second, notice that the vector field $\omega$-dual to $\theta + t df$ is a convex combination of the two outward pointing vectors fields $Z_\theta$ and $Z_{\theta+df}$, and so is also outward pointing. Thus $\{ (\mathring{W}, \theta+ tdf) \}_{t \in [0,1]}$ is a smooth family of Liouville domains.

The cylindrical completions of $\{ (\mathring{W}, \theta+ tdf) \}_{t \in [0,1]}$ give a smooth family of finite type complete Liouville manifolds, say $\{ (W, \theta_t) \}_{t \in [0,1]}$. 
This is automatically a convex symplectic deformation in the sense of \cite[Section 2]{Seidel-Smith}. 
By \cite[Lemma 5]{Seidel-Smith}, there is a smooth family of diffeomorphisms $h_t: W \to W$, with $h_0 = \id$, such that $h_t^\ast \theta_t = \theta_0 + dg_t$, where $(t,x) \mapsto g_t(x)$ is a compactly supported function on $[0,1]\times W$. Inspection of the proof shows that the family $h_t$ is canonical up to isotopy through diffeomorphisms with the same properties.

This implies that we have a trivial inclusion of Liouville manifolds $(W, \theta) \hookrightarrow (W, \theta_1)$ in the sense of \cite[Definition 2.4]{GPS1}. By \cite[Lemma 3.41]{GPS1} (see also \cite[Lemma 3.4]{GPS2}), there is an induced equivalence $c: \cW_{\pr}(W, \theta) \to \cW_{\pr}(W, \theta_1)$. 
Postcomposing $\phi: \cW_{\pr}(\mathring{W}, \theta) \to \cW_{\pr}(\mathring{W}, \phi^\ast \theta)$ with the inverse to $c$ gives an autoequivalence of $\cW_{\pr}(\mathring{W}, \theta)$, which we will also denote $\phi$. 

To ensure that we obtain a group homomorphism, we require a consistency of continuation maps. This  follows from the compatibility of the Liouville inclusions obtained from \cite[Section 2]{Seidel-Smith} with concatenation of families of finite type complete Liouville manifolds.

Finally,
if $\phi$ and $\phi'$ are in the same connected component of $\Symp_{\ex} (\mathring{W}, \theta)$, they are related by a Hamiltonian isotopy. Furthermore, if $\phi$ is Hamiltonian, the maps constructed in the two steps of the argument above are evidently quasi-inverse. This completes the proof.
\end{proof}

\begin{lemma}cf. \cite[Lemma 1.1]{BEE}. \label{lem:retraction-to-exact-symplectos}
The inclusion $\Symp_{\ex} (\mathring{W}, \theta) \hookrightarrow \Symp (\mathring{W})$ is a weak homotopy equivalence. 

\end{lemma}

\begin{proof}
An analogous statement for complete Liouville manifolds is contained in \cite[Lemma 1.1]{BEE}. This shows that given an arbitrary symplectomorphism $\psi: (W, \theta) \to (W', \theta')$ between finite type complete Liouville manifolds, there exists a symplectic isotopy $\sigma_t: (W, \theta) \to (W, \theta)$, $t \in [0,1]$, $\sigma_0 = \id$, such that $\psi \circ \sigma_1$ is exact. 
Inspecting the proof, we see that the isotopy is canonical up to deformation within the class with the same properties, and that the argument can be carried out for families of symplectomorphisms parametrised by arbitrary compact cell complexes. 

We now want to upgrade this to a statement for Liouville domains. We start with the argument for a single map.
Suppose $\psi: \mathring{W} \to \mathring{W}$ is an arbitrary symplectomorphism. 
Tautologically, this induces a symplectomorphism from the conical extension of $(\mathring{W}, \theta)$, say $(W, \theta)$, to the conical extension of $(\mathring{W}, \psi^\ast \theta)$, say $(W', \theta')$. Let us also denote this symplectomorphism by $\phi$. By construction, the cylindrical end $E = [0,\infty) \times \partial \mathring{W}$ for $(W, \theta)$ is taken to the one for  $(W, \psi^\ast \theta)$, with $\psi$ preserving $s$-level contact shells, $s \in [0, \infty)$, and $\psi \circ \rho_s = \rho'_s \circ \psi$, where $\rho_s$ is the Liouville flow for $(W, \theta)$ and $\rho_s'$ the one for $(W', \theta')$. 

We now apply \cite[Lemma 1.1]{BEE} to get a symplectic isotopy $\sigma_t$, $t \in [0,1]$, as above. 
Inspecting their proof, we see that outside a compact set, $\sigma_t$ is given by integrating a vector field of the form $e^{-s} \tilde{V}$, where $\tilde{V}$ is independent of $s$. In particular, for sufficiently large $s$, the contact level-set $\{ s \} \times \partial \mathring{W}$ is displaced by $C^\infty$-small amounts under $\sigma_t$, for any $t \in [0,1]$. (The statement of their lemma only spells this out for $C^0$.) On the other hand, under $C^\infty$ small perturbations, both the contact condition and the condition that a vector field be outward pointing are open. 
Fix such a sufficiently large $s$. For any $t \in [0,1]$, let 
$\mathring{W}_t := W \backslash \sigma^{-1}_t ( [s, \infty) \times \partial \mathring{W})$.
By construction, $(\mathring{W}_t, \theta)$ is a Liouville domain with conical completion $W$, and is canonically Liouville isomorphic (via $\sigma_t$) to $(\mathring{W} \cup ([0,s] \times \partial\mathring{W}), \theta)$. 

Altogether, we have deformed $\psi$ to an exact symplectomorphism of  $(\mathring{W} \cup ([0,s] \times \partial\mathring{W}), \theta)$. Conjugating by the time $s$ Liouville flow now gives the desired map for $\psi$. 
Finally, for a family of maps $\{ \psi_b \}_{b \in B}$ indexed by a compact cell complex $B$, one can choose the parameter $s$ uniformly over $B$. Together with the observation that the argument in \cite{BEE} works in families, this completes the proof.
\end{proof}

\begin{corollary}\label{cor:symplectos-act}
The group
$\pi_0 \Symp (\mathring{W})$ acts on $\cW_{\pr}(W, \theta)$. 
\end{corollary}

\begin{proof}
This now follows immediately from combining Lemmas \ref{lem:exact-action} and \ref{lem:retraction-to-exact-symplectos}.
\end{proof}

\subsection{Gradings and orientations}\label{sec:gradings-and-orientations}

We now assume that $2c_1(W, d \theta) = 0 \in H^2 (W; \bZ)$, i.e.~the bundle $(\Lambda^{\dim{W}} T^\ast W)^{\otimes 2}$ is trivial (here $\dim$ denotes the complex dimension). Then $W$ admits a grading, i.e.~a choice of fibrewise universal cover for the Lagrangian Grassmanian bundle $\Lag (W) \to W$, say $\widetilde{\Lag} (W) \to W$.  The isomorphism classes of such covers are in one-to-one correspondence with trivialisations of $(\Lambda^{\dim{W}} T^\ast W)^{\otimes 2} \simeq W \times \bC$, and form an affine space over $H^1(W; \bZ)$ (see \cite[Lemma 2.2]{Seidel_graded}). 

Fix a grading on $W$. Given a Lagrangian $L \subset W$, the trivialisation $(\Lambda^{\dim{W}} T^\ast W)^{\otimes 2} \simeq W \times \bC$ determines a (homotopy class of) map $L \to \bC^\ast$, and so a class in $H^1(L; \bZ)$, called the Maslov class of $L$. There exists a lift of $TL \subset \Lag (W)$ to $\widetilde{\Lag} (W)$ precisely when $L$ has Maslov class zero; a choice of such a lift is called a grading on $L$. 

We can now upgrade to the `full' version of the wrapped Fukaya category, which will be used for the rest of this article.

\begin{definition}
The wrapped Fukaya category $\cW(W, \theta)$ is defined using the set-up in \cite[Section 2.4]{GPS2}, with the following choices: coefficient ring $\bZ$; Lagrangians are exact, cylindrical, Maslov zero and spin, and, as objects of $\cW(W, \theta)$, equipped with  brane data comprising a $Spin$-structure and a  grading.  

The compact Fukaya category $\scrF(W, \theta)$ is the subcategory of compact Lagrangian submanifolds. 
\end{definition}
We will often suppress $\theta$ from the notation when it is clear from the context which primitive we are referring to. 

\begin{remark} For homological mirror symmetry statements, we'll base change the ground ring for $\cW(W)$ and $\scrF(W)$ from $\bZ$ to $\bC$; as this will be clear from context, we will also denote those categories by $\cW(W)$ and $\scrF(W)$. 
\end{remark}

Suppose that we are given $\phi \in \Symp (W)$. We say that $\phi$ is gradeable if under pullback by $\phi$, our choice of trivialisation of $(\Lambda^{\dim{W}} T^\ast W)^{\otimes 2}$ is preserved; a grading for $\phi$ is then a choice of lift of the pullback map to a bundle automorphism of $\widetilde{\Lag}(W)$. We let $\Symp^{\gr} W$ denote the group of graded symplectomorphisms of $W$, and $\Symp^{\gr}_{\ex} (W, \theta)$ the subgroup  of exact graded symplectomorphisms. From the discussion above there's an exact sequence\footnote{The final map is not a group homomorphism, but its kernel is nonetheless a subgroup of $\Symp(W)$. } (see \cite[Lemma 2.4]{Seidel_graded})
\begin{equation} \label{eqn:gradings-exact-seq}
1 \to \bZ \to \Symp^{\gr} (W) \to \Symp (W) \to H^1(W; \bZ).
\end{equation}

If $\phi  \in \Symp(W)$ is gradeable and has compact support, it comes with a preferred grading: the one which fixes $\widetilde{\Lag}(W)$ outside of the support of $\phi$.
This implies that there is a splitting of the forgetful map from $\Symp_c^{\gr}(W)$ to the group of gradeable compactly supported symplectomorphisms. 

The results of the previous subsection readily generalise as follows.

\begin{lemma}\label{lem:exact-action-graded}
The group $\pi_0 \Symp_{\ex}^{\gr} (\mathring{W}, \theta)$ acts on $\cW(W, \theta)$. 
\end{lemma}

\begin{proof}
Given $\phi \in \Symp_{\ex}^{\gr} (\mathring{W}, \theta)$, we get a map $\cW(\mathring{W}, \theta) \to \cW(\mathring{W}, \phi^\ast \theta)$ by pulling back all Floer data, taking a Lagrangian $L$ to $\phi(L)$, and using the choice of grading for $\phi$ to determine the grading for $\phi(L)$ (the spin structure is just pulled back). 

We can now follow the proof of Lemma \ref{lem:exact-action}, noting that in any given connected component of $\Symp_{\ex} (\mathring{W}, \theta)$, either all maps are gradeable or none of them are. 
\end{proof}

\begin{lemma}\label{lem:retraction-to-graded-exact-symplectos}
The inclusion $\Symp_{\ex}^{\gr} (\mathring{W}, \theta) \hookrightarrow \Symp^{\gr}(\mathring{W})$ is a weak homotopy equivalence. 
\end{lemma}

\begin{proof}
This readily follows from Lemma \ref{lem:retraction-to-exact-symplectos}, by noticing, first, that for any given connected component of $\Symp(\mathring{W})$, resp.~$\Symp_{\ex} (\mathring{W}, \theta)$, either all maps in it are gradeable or none of them are; and that for each such gradeable component there are $\bZ$ components in  $\Symp^{\gr}(\mathring{W})$, resp.~$\Symp_{\ex}^{\gr} (\mathring{W}, \theta)$.
\end{proof}

We then immediately get the following:

\begin{corollary}\label{cor:graded-symplectos-act}
The group $\pi_0 \Symp^{\gr} (\mathring{W})$ acts on $\cW(W, \theta)$. 
\end{corollary}

The natural module action of symplectic cohomology on wrapped Floer cohomology shows that $\cW(W)$ admits the structure of a category linear over $SH^0(W)$. From the discussion above, we immediately get that graded symplectomorphisms act on $\cW(X)$ by $SH^0(W)$-linear equivalences.

\section{The conifold smoothing} \label{sec:symplectic-side}

\subsection{Descriptions of the conifold smoothing $X$}

Consider $\bC^2\times\bC^2\times\bC^*$ with co-ordinates $(u_1,v_1,u_2,v_2,z)$, equipped with the standard product symplectic form $\omega = d\theta$. Here $\bC^\ast$ has the symplectic form identifying it with $T^\ast S^1$. For concreteness, take the K\"ahler form with potential $|u_1|^2+ |u_2|^2+ |v_1|^2+|v_2|^2+(|z|^2 + 1/|z|^2)$. (The last term can be thought of as the restriction of the standard potential on  $\bC^2$ to $\{ zw =1\}$.)



The \emph{conifold smoothing} $X$ is the symplectic submanifold defined by the equations
\begin{equation} \label{eqn:affine}
X = \{u_1v_1 = z-1, \ u_2v_2 = z+1\}.
\end{equation}
It is the fibre $X=X_{1,-1}$ over $(+1,-1) \in \ (\bC^* \times \bC^*) \backslash \Delta$ of a family of symplectic submanifolds
\begin{equation}
X_{a,b} = \{u_1v_1 = z-a, \ u_2v_2 = z-b\}. \label{eq:X_{a,b}}
\end{equation}
Let $\pi: X_{a,b} \to \bC^*$ denote the projection to the $z$ co-ordinate. This is a Morse-Bott-Lefschetz fibration, meaning that $d\pi$ has transversely non-degenerate critical submanifolds.  The smooth fibres are isomorphic to $\bC^*\times\bC^* = T^*T^2$, and there are singular fibres $(\bC\vee\bC) \times \bC^*$ respectively $\bC^* \times (\bC\vee\bC)$ over $a$ respectively $b$.  

\begin{lemma}\label{lem:parallel-transport} 
For fixed $a,b$, symplectic parallel transport for the fibration $\pi: X_{a,b} \to \bC^*$ is globally well-defined.
\end{lemma} 

\begin{proof}
There is a Hamiltonian $T^2$-action on $X_{a,b}$, rotating the $(u_i,v_i)$-planes, and $\pi$ is a $T^2$-equivariant map. Given any symplectic fibration with a fibrewise Hamiltonian action of a Lie group $G$, each component of the moment map for the $G$-action is preserved by symplectic parallel transport (cf. \cite[Section 5]{Auroux:Gokova} or \cite[Lemma 4.1]{SW}). In the case at hand, the moment map is fibrewise proper; the result follows.
\end{proof}

\begin{lemma}\label{lem:matching-spheres}
Any matching path $\gamma$ between the two critical values in the base  $\bC^*$ of the Morse-Bott-Lefschetz fibration defines a Lagrangian sphere $S_\gamma$ in $X$.
\end{lemma}

\begin{proof}
Paths emanating from $+ 1\in \bC^*$, resp.~$-1 \in \bC^\ast$ have associated Morse-Bott thimbles $S^1\times D^2$, respectively $D^2\times S^1$. These restrict to Lagrangian $T^2$s in nearby fibres. 
Let $\mu$ be the moment map generating the $T^2$ action. Then each thimble is the intersection of $\mu^{-1}(0)$ with the preimage of an open subset of $\gamma$, and we immediately see that the thimbles match without any corrections. We get 
 $$S_\gamma  = \mu^{-1}(0) \cap \pi^{-1} (\gamma). $$
To see that the $S_\gamma$ are all spheres,  first notice if $\gamma$ is the upper or lower unit circle, then $S_\gamma$ is presented via the standard genus one Heegaard splitting of the sphere: the vanishing cycles give the standard generators of $\bZ^2 = H_1(T^2;\bZ)$. Now for a general matching path, the Morse-Bott monodromies act trivially on the homology of the fibre $T^*T^2$, hence preserve the fact that the two vanishing cycles at the ends of the path span $\bZ^2 = H_1(T^2;\bZ)$.
\end{proof}

Recall that the compact core (or `skeleton') of a Liouville manifold $W$ with Liouville vector field $Z$ is the locus $\cup_{K\subset W} \cap_{t>0} \phi_Z^{-t}(K)$, where we take the union over relatively compact codimension zero subdomains with smooth boundary transverse to $Z$. When $W$ is finite type, so the completion of a fixed subdomain $K$,  this is just  $\bigcap_{t>0} \phi_Z^{-t}(K)$, and can be understood as the locus of points which do not flow to infinity under the Liouville vector field.

\begin{lemma} \label{lem:hopf}
$X_{a,b}$ is naturally a Weinstein manifold, with compact core $S^3 \cup_{\mathrm{Hopf}} S^3$.
\end{lemma}

\begin{proof}
It suffices to consider $X=X_{1,-1}$.  Recall that for any K\"ahler potential $\rho$, the Liouville vector field dual to the Liouville form $-d^c \rho$  is the gradient vector field of $\rho$ with respect to the K\"ahler metric.  We will use  the K\"ahler potential $\rho = |u_1|^2+ |u_2|^2+ |v_1|^2+|v_2|^2+2(|z|^2 + 1/|z|^2)$, where the factor of 2 is benign but simplifies some calculations. We have Liouville form $\theta=-d^c \rho$; let $Z$ be the Liouville vector field dual to $\theta$. 

At any point away from the singular locus of $\pi$, the tangent space splits as a direct sum of the tangent space to the fibre $\pi^{-1}(\pt)$ and its symplectic orthogonal. 
This induces a decomposition of the Liouville vector field, say $Z = Z_F + Z_B$. 

Let $ i: F \hookrightarrow X$ be the inclusion of a fibre. At any smooth point of $i(F)$, the vector field $Z_F$ is the dual (using $i^\ast \omega$) to $i^\ast \theta$. Also, we have that $i^\ast \theta = i^\ast (-d^c(|u_1|^2+ |u_2|^2+ |v_1|^2+|v_2|^2))$. In particular, a standard calculation shows that $Z_F$ vanishes whenever $|u_i| = |v_i|$ for both $i$, and that its positive flow (on a fixed fibre) scales $\log (|u_i| / |v_i|)$ for both $i$. 

Symplectic parallel transport intertwines the fibrewise $T^2$ action, preserving the ratios $|u_1| / |v_1|$ and $|u_2|/ |v_2|$. Thus the flow of $Z$ itself, while not fibre-preserving, acts by scaling $\log (|u_i| / |v_i|)$ for both $i$. In particular, any point in $X$ which does not satisfy $|u_i| =|v_i|$ for both $i$ flows off to infinity, and so the compact core is contained in the subset $K = \{ |u_i| = |v_i|, \, i=1, 2 \} \subset X$.  
(Note this also applies to points in the annuli of critical points with $|u_i| \neq|v_i|$ for one $i$.)

It remains to analyse the Liouville flow on $K$.  Using $T^2$-equivariance, one sees that the degenerate 2-form $\omega|_K$ is 
pulled back from a symplectic form on the base, say $\omega_B$.  
Restricted to $K$, the K\"ahler potential is
$$
\rho (z) = 2|z-1| + 2|z+1| + 2(|z|^2 + 1/|z|^2).
$$
The Liouville flow $Z=Z_B$ is the gradient vector field of $\rho(z)$ with respect to the restriction of the K\"ahler metric to $K$. This is the same as the pullback of the gradient flow of $\rho(z)$ on $\bC^\ast$, using the metric associated to $\omega_B$.  

Irrespective of the metric used for geodesic flow, standard calculus techniques show that the real function $\rho(z)$ has two minima, for $z = \pm 1$, and two saddles at opposite points on the imaginary axis, with modulus the real root of $x^3+x^2-1$ (approx.~0.75).  

Possibly after a small perturbation of $\omega_B$ (and so of the metric on the base), there is a unique trajectory (arc) from each saddle to each minimum, which can be parametrised to have strictly monotone angular coordinate. Put together, the two arcs emanating from one saddle will give the matching path for one core, and the other two arcs will give the matching path for the other core.  (The constructed vector field is not necessarily Morse; that can either be achieved by a further small perturbation of the K\"ahler potential, alternatively one can note that the definition of the skeleton or compact core doesn't rely on the Morse condition, see e.g. \cite{Starkston} for a general discussion of skeleta in a Morse-Bott setting.)
\end{proof}

It follows that $X$ is the Weinstein completion of a Weinstein domain obtained by plumbing two copies of the disc cotangent bundle $D^*S^3$ along a Hopf link.

An affine variety determines a (complete) Weinstein symplectic manifold canonically up to exact symplectomorphism \cite[Section 4b]{Seidel:biased}.  We will give two further descriptions of $X$ by giving two further descriptions of the affine variety \eqref{eqn:affine}.

Recall that if $Z \subset \bC^2$ is a Zariski open subset and $f: Z \to \bC$ is a regular function, the \emph{spinning} of $f$ is the 3-fold
\begin{equation} \label{eqn:spinning}
\{(x,y,u,v) \in Z \times \bC^2 \, | \, f(x,y) = uv\}.
\end{equation}
This is a fibration over $Z$ by affine conics $\bC^*$, with singular fibres $\bC\vee\bC$ along $f^{-1}(0) \subset Z$.  

\begin{lemma}
$X$ can be obtained by spinning $Z = \bC^2 \backslash \{xy=1\}$ along the map $f(x,y) = xy-2$.
\end{lemma}  

\begin{proof}
Let $g: \bC^2 \to \bC$ be given by $g(x,y) = xy-1=z$, so $Z = \bC^2 \backslash g^{-1}(0)$ and $g: Z \to \bC^*$ has a Lefschetz singular fibre over $z=-1$. The spinning is then a $\bC^*$-fibration over $\bC^2\backslash \{xy=1\}$ with singular fibres over $\{xy-1=1\}$.  It is globally cut out of $\bC^2\times\bC^2\times\bC^*$, with co-ordinates $(x,y,u,v,z)$, by the equations $xy=z+1$ and $uv=z-1$.
\end{proof}

$Z$ contains an immersed Lagrangian $2$-sphere, obtained from parallel transport around the unit circle in the $z$-plane, cf.~\cite[Section 11]{Seidel:categorical_dynamics}. Viewing this as a union of two Lefschetz thimbles fibred over the upper and lower half-circles with common boundary in the fibre $z=1$, the immersed $2$-sphere naturally `spins' to the compact core $S_0 \cup_{\mathrm{Hopf}} S_1$ of $X$ described above, i.e. the matching $3$-spheres are obtained as $S^1$-fibrations over the thimbles with fibres degenerating over the boundary of the $2$-disc. 

Our final description of $X$ is as a log Calabi-Yau 3-fold. We won't use this directly, but it may be of independent interest. Consider $\bP^1 \times \bP^1 \times \bP^1$, which we view as the trivial $\bP^1\times \bP^1$-bundle over the base $\bP^1$. We blow up this variety along the (disjoint) curves $\bP^1 \times \{pt\} \times \{1\}$ and $\{pt\} \times \bP^1 \times \{-1\}$ to obtain a variety $\overline{X}$ with an induced map $\overline{X} \to X \to \bP^1$. Let $D$ denote the divisor given by the union of two smooth fibres of $\overline{X} \to \bP^1$ together with the subset of the toric boundary comprising its four components isomorphic to the first Hirzebruch surface $\bF_1$, see Figure \ref{fig:compactification-of-X}.

\begin{figure}[ht]
\begin{center}

\begin{tikzpicture}

\newcommand{\Depth}{2}
\newcommand{\Height}{2}
\newcommand{\Width}{2}
\coordinate (O) at (0,0,0);
\coordinate (A) at (0,\Width,0);
\coordinate (B) at (0,\Width,\Height);
\coordinate (C) at (0,0,\Height);
\coordinate (D) at (\Depth,0,0);
\coordinate (E) at (\Depth,\Width,0);
\coordinate (F) at (\Depth,\Width,\Height);
\coordinate (G) at (\Depth,0,\Height);

\draw[semithick] (O) -- (C) -- (G) -- (D) -- cycle;
\draw[semithick] (O) -- (A) -- (E) -- (D) -- cycle;
\draw[semithick] (O) -- (A) -- (B) -- (C) -- cycle;
\draw[semithick] (D) -- (E) -- (F) -- (G) -- cycle;
\draw[semithick] (C) -- (B) -- (F) -- (G) -- cycle;
\draw[semithick] (A) -- (B) -- (F) -- (E) -- cycle;
\draw[ultra thick] (A) -- (B);
\draw[ultra thick] (F) -- (G);

\coordinate (P) at (\Depth-0.4,\Width,\Height);
\coordinate (Q) at (\Depth,\Width,\Height-0.6);
\coordinate (R) at (\Depth-0.4,0,\Height);
\coordinate (S) at (\Depth,0,\Height-0.6);
\coordinate (U) at (0.4,\Width,0);
\coordinate (V) at (0.4,\Width,\Height);
\coordinate (W) at (0, \Width-0.6,\Height);
\coordinate (Z) at (0,\Width-0.6, 0);

\draw[fill=green!20, opacity=0.8] (P) -- (Q) -- (E) -- (U) -- (V) --(P);
\draw[fill=red!20, opacity=0.8] (V) -- (W) -- (C) -- (R) -- (P);
\draw[gray,semithick] (Q) -- (S);
\draw[gray,semithick] (R) -- (S);
\draw[gray,semithick, dotted] (W) -- (Z);
\draw[gray,semithick, dotted] (U) -- (Z);

\end{tikzpicture}
\end{center}
\caption{The toric compactification of $X$; blow up the thickened black edges on the cube for $\bP^1\times \bP^1 \times \bP^1$. This slices off wedges of the corresponding moment polytope. Two of the four $\bF_1$-boundary components of the result have been shaded.}

\label{fig:compactification-of-X}

\end{figure}
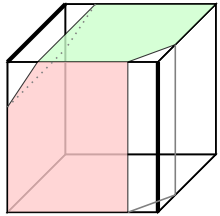

\begin{lemma}
$X$ is symplectically equivalent to the symplectic completion of $\overline{X}\backslash D$.\end{lemma}

\begin{proof}
In an affine chart $\bC_p \times \bC_q \times \bC^*_z$ we blow up $\{p=1=z\}$ and $\{q=1=-z\}$. Taking $[\lambda,\mu] \in \bP^1$ and $[\tilde{\lambda}, \tilde{\mu}] \in \bP^1$ to be homogeneous co-ordinates on two different copies of $\bP^1$, the blow-up is given by 
\[
(p-1)/(z-1) = \lambda / \mu; \ (q-1)/(z+1) = \tilde{\lambda} / \tilde{\mu}.
\] 
Letting $u_1 = q-1$, $u_2 = p-1$, $v_1 = \tilde{\mu}/\tilde{\lambda}$ and $v_2 = \mu/\lambda$ gives
\[
u_1v_1 = z+1, \ u_2v_2 = z-1;
\]
since $z\in\bC^*$ none of $u_i, v_j$ can be infinite. Now remove the irreducible  divisor from the projective blow-up where any of the $u_i, v_j$ co-ordinates do become infinite. This meets the general fibre in its toric boundary; each special fibre is again toric, isomorphic to $(\bP^1\times \bP^1) \cup_{\bP^1 \times \{pt\}} (\bP^1\times \bP^1)$, and the divisor meets this in a hexagon of lines (one does not remove the entire line of intersection of the two components, but only its intersection with the rest of the boundary). 
\end{proof}

\subsection{Exact Lagrangian submanifolds\label{Subsec:exact-lagrangian-submanifolds}}

As remarked previously in Lemma \ref{lem:matching-spheres}, any matching path $\gamma$ between the two critical values in the base of the Morse-Bott-Lefschetz fibration $X \to \bC^*$ defines a Lagrangian sphere $S_\gamma$ in the total space $X$.  
We will later be particularly interested in the collection of  matching spheres $\{S_i\}_{i\in\bZ}$ corresponding to the matching paths given in Figure \ref{fig:spherical-conventions}.

\begin{figure}[htb]
\begin{center}
\includegraphics[scale=0.25]{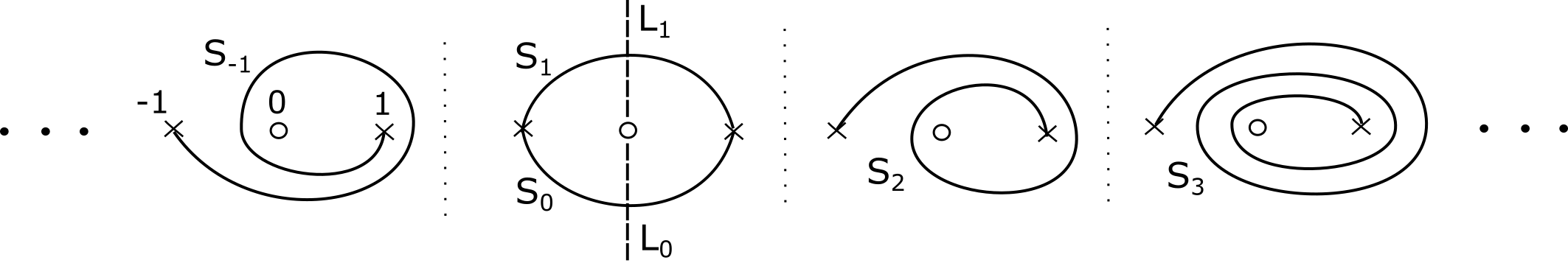}
\caption{Lagrangian matching spheres $S_i$ (with matching paths $\gamma_i$) and Lagrangian discs $L_i$ (see Theorem \ref{thm:hms-equivalence}).}
\label{fig:spherical-conventions}
\end{center}
\end{figure}

Suppose we are given an embedded $S^1 \subset \bC^\ast \backslash \{ -1, 1 \}$ in the base of the Morse-Bott-Lefschetz fibration, avoiding the critical values. 
By similar considerations to those in  Lemmas \ref{lem:parallel-transport} and \ref{lem:matching-spheres}, parallel transporting the $T^2$ vanishing cycle about the $S^1$ gives a Lagrangian $T^3$. This $T^3$ will be exact precisely when the $S^1$ in the base is Hamiltonian isotopic to the unit circle in $\bC^\ast$. (Recall here that we are using the standard K\"ahler form for $\bC^\ast$ as an affine manifold.) For any choice of such circle containing all of $\{ -1, 0 , 1 \}$, call the associated exact torus $T$.

\subsection{Gradings\label{Subsec:gradings}}

As $X$ is an affine complete intersection in Euclidean space, we have that $2c_1(X) = 0$, and the isomorphism classes of gradings on $X$ form an affine space over $H^1(X; \bZ) \cong \bZ$. 
Our convention is that throughout this paper we work with a fixed, distinguished grading for $X$: the unique one for which the exact Lagrangian torus $T$ defined above has Maslov class zero (this will also be true of any other tori fibred over $S^1$s containing $0$ in the base of the Morse-Bott-Lefschetz fibration).

We now want to fix gradings for the $S_i$.
Fix an arbitrary grading for $S_0$. For $i \neq 0$, we want to grade $S_i$ as follows: first, note that we can make fibrewise Morse perturbations such that $S_0$ and $S_i$ intersect transversally in $4|i|$ points. One can check that there is a choice of grading for $S_i$ such that: for $i>0$, as generators for the Floer complex $CF^\ast (S_0, S_{-i})$, there are $i+1$ intersection points with grading zero; $2i$ points with grading one; and $i-1$ points with grading two; and for $i<0$, the same is true for $CF^\ast (S_i, S_0)$.

\begin{remark}
These gradings are chosen with mirror symmetry in mind: $S_i$ will be mirror to $\calO_C(-i)$ for a fixed $(-1,-1)$ curve $C$ in a threefold $Y$, see Section \ref{sec:main-theorem}.
\end{remark} 

We fix an arbitrary grading on the Lagrangian torus $T$.

\subsection{Symplectomorphisms of $X$: first considerations}

Let $(\mathring{X}, \theta)$ be a Liouville domain such that $(X, \theta)$ is its associated completion. 
Following the notation in Section \ref{sec:symplecto-actions}, let $\Symp(\mathring{X})$ denote the group of all symplectomorphisms of $\mathring{X}$ equipped with the $C^\infty$ topology, and let  $\Symp_{\ex}(\mathring{X}, \theta)$ be the subgroup of exact symplectomorphisms. We define graded versions $\Symp^{\gr}(\mathring{X})$ and $\Symp^{\gr}_{\ex} (\mathring{X})$ similarly.

By Lemma \ref{lem:retraction-to-graded-exact-symplectos}, the inclusion $\Symp_{\ex}^{\gr}(\mathring{X}, \theta) \hookrightarrow \Symp^{\gr}(\mathring{X})$ is a weak homotopy equivalence. We define  $\Symp^{\gr}(X)$ to mean $\Symp^{\gr} (\mathring{X})$, and $\Symp_{\ex}^{\gr} (X)$ to mean $\Symp_{\ex}^{\gr} (\mathring{X}, \theta)$. (Similarly for ungraded versions.) 
For any other Liouville domain $(\mathring{X}', \theta)$ with completion $(X, \theta)$, there is a natural weak homotopy equivalence $ \Symp^{\gr}( \mathring{X}) \to \Symp^{\gr}(\mathring{X}')$, and similarly for exact maps. In particular, the group $\pi_0 \Symp^{\gr}(X)$ is independent of choices. By Corollary \ref{cor:graded-symplectos-act}, $\pi_0 \Symp^{\gr}(X)$ acts on the Fukaya category $\cW(X)$; this readily restricts to an action on $\scrF(X)$. 

Let $\Symp_c (X)$ denote the group of compactly supported symplectomorphisms of $X$. 
 
\begin{lemma}
Any element $\phi \in \Symp_c (X)$ is  strongly exact: there exists a compactly supported function $f$ such that $\phi^\ast \theta = \theta + df$, where $\theta$ is the Liouville form on $X$. 
\end{lemma}

\begin{proof}
As $\phi$ is a symplectomorphism, $\phi^\ast \theta - \theta$ is a closed one-form; moreover, away from a compact set, $\phi$ is the identity, and so $\phi^\ast \theta - \theta$ vanishes there. This means it defines a class in $H^1_c(X; \bR)$. By Poincar\'e duality, this is isomorphic to $H_5(X; \bR)$, which vanishes as $X$ has the homotopy type of a 3-cell complex. Thus $\phi^\ast \theta - \theta = df$ for some function $f$, locally constant outside the domain of $\phi$. Finally, as the boundary at infinity of $X$ is connected, we see that we can assume that $f$ vanishes outside a compact set. 
\end{proof}

\begin{lemma} \label{lem:compact-support-implies-graded}
Any element $\phi \in \Symp_c (X)$ is gradeable. Morever, it admits a preferred grading; this determines a splitting of the forgetful map $\Symp^{\gr}_c(X) \to \Symp_c(X)$. 
\end{lemma}

\begin{proof} See \cite[Remark 2.5]{Seidel_graded}. 
As $\phi$ has compact support, the obstruction to grading it is a map $h \in [X, \bC^\ast]$ such that $h \equiv 1 $ outside a compact set, i.e.~a class in $H^1_c(X; \bZ)$. As this vanishes, any such  $\phi$ is gradeable. 
Finally, as noted in Section \ref{sec:gradings-and-orientations}, any gradeable compactly supported symplectomorphism has a distinguished grading: the one which is the identity on $\widetilde{\Lag}(X)$ outside a compact set. 
\end{proof}

It is then immediate that $\pi_0 \Symp_c(X)$ acts on $\cW(X)$ and $\scrF(X)$. 

There is a well-known source of compactly supported symplectomorphisms:  for any matching path $\gamma$ in the base of the Morse-Bott-Lefschetz fibration, the Dehn twist  $\tau_{S_\gamma} \in \pi_0 \Symp_c X$. 
By  \cite[Lemma 4.13]{SW}, $\tau_{S_\gamma}$ is the monodromy associated to a full right-handed twist in the path $\gamma \subset \bC^\ast$. More precisely, suppose we parametrise the full right-handed twist as a one-parameter family of compactly supported diffeomorphisms of $\bC^\ast$, giving a map 
$\rho: [0,1] \times \bC^\ast \to \bC^\ast$.
See Figure \ref{fig:fullRHtwist} for an illustration.
\begin{figure}[htb]
\begin{center}
\includegraphics[scale=0.30]{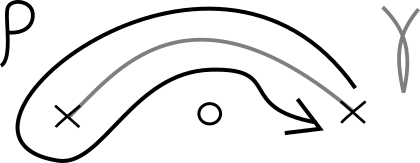}
\caption{For a choice of $\gamma$, visualisation of the associated full right-handed twist $\rho$, as an isotopy of $\bC^\ast$.}
\label{fig:fullRHtwist}
\end{center}
\end{figure}

Then $\tau_{S_\gamma}$ can be interpreted as the monodromy of the family $X_{\rho(t,-1), \rho(t,1)}$, $t \in [0,1]$, with the notation from Equation \ref{eq:X_{a,b}}. We will revisit this in Lemma \ref{lem:PBr_3-action-ungraded}.

\subsection{The $\PBr_3$ action}

We start with some generalities on the pure braid group $\PBr_3$. We will usually think of it as a subgroup of the fundamental group of the space of configurations of two points in $\bC^*$, with basepoint $\{ -1, 1 \} \subset \bC^\ast$; equivalently, it's naturally isomorphic to $\pi_1 ((\bC^\ast)^2 \backslash \Delta )$. 

By  \cite[Theorem 2.3]{Margalit-McCammond}, we have a presentation for $\PBr_3$ with  three generators $R_1, R_2$ and $R_3$,  given by the full twists described in Figure \ref{fig:PBr_3-generators}, and the sole relation
$$
R_1 R_2 R_3 = R_2 R_3 R_1 = R_3 R_1 R_2.
$$
The element $R_1 R_2 R_3$ generates the centre $Z(\PBr_3) \cong \bZ$. Geometrically, it corresponds to an inverse Dehn twist in a boundary $S^1$.

\begin{figure}[htb]
\begin{center}
\includegraphics[scale=0.25]{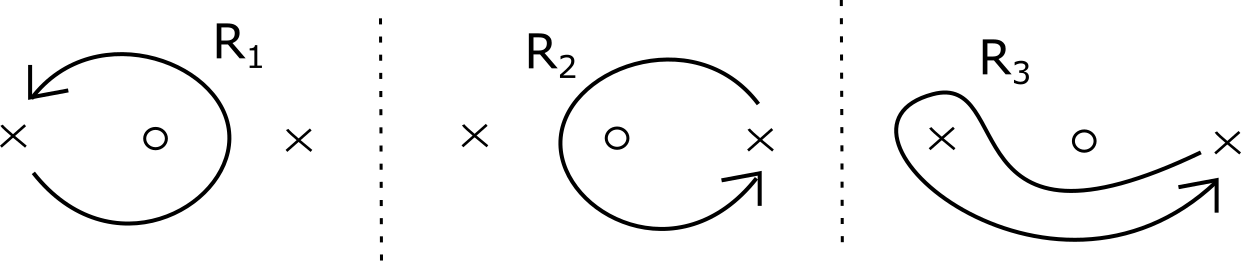}
\caption{The generators for $\PBr_3$}
\label{fig:PBr_3-generators}
\end{center}
\end{figure}

Quotienting by the center, we get isomorphisms 
$$\PBr_3 \simeq \PBr_3 / Z(\PBr_3) \times Z(\PBr_3) \simeq (\bZ \ast \bZ) \times \bZ.
$$
This decomposition as a direct product of groups is of course not unique: for instance we could have $(\bZ\langle R_i \rangle \ast \bZ \langle R_j \rangle) \times \bZ \langle R_1 R_2 R_3 \rangle$ for any $i \neq j$.

Given any matching path $\gamma$ between $-1$ and $+1$ in $\bC^\ast$, consider the full twist $t_\gamma \in \PBr_3$. We want to characterise the subgroup generated by these twists. 

\begin{proposition} \label{prop:compact-braid-action}
Let $\PBr_3^c \vartriangleleft \PBr_3$ be the subgroup of $\PBr_3$ generated by the twists $t_\gamma$. 
Then:

(i) Fix $\psi \in \PBr_3$. Let $b_{-1}$, $b_0$ and $b_1$ be the strands based at $-1$, $0$ and $1$ respectively. Then $\psi$ lies in $\PBr_3^c$ if and only if the winding number of $b_{-1}$ about $b_0$ and the winding number of $b_1$ about $b_0$ are both zero. In other words, 
$$
\PBr_3^c = \ker \{ \PBr_3 \xrightarrow{(\text{forget $b_{-1}$, forget $b_1$}) } \PBr_2 \times \PBr_2 \}
$$
where $\PBr_2$ is thought of as the fundamental group of the configuration space of one point in $\bC^\ast$, based at $\{ 1\}$ for the first factor and at $\{ -1\}$ for the second. 

(ii) We have an isomorphism $\PBr_3^c \cong \bZ^{\ast \infty}$, with generators the full twists $t_i$ in the matching paths $\gamma_i$ for $S_i$ from Figure \ref{fig:spherical-conventions}.

\end{proposition}

\begin{proof}

Given any $\gamma$, it is immediate that $t_\gamma \in \ker \{ \PBr_3 \xrightarrow{(\text{forget $-1$, forget 1})}  \PBr_2 \times \PBr_2 \}$, and so the inclusion one way is clear. 

Now suppose that we're given a pure braid 
$$
\psi \in K = \ker \{ \PBr_3 \xrightarrow{(\text{forget $b_{-1}$, forget $b_1$)}}  \PBr_2 \times \PBr_2 \}. 
$$

To analyse this, first consider simply forgetting $b_1$. 
For any element in the kernel, $b_{-1}$ and $b_0$ can be simultaneously pulled taut, so that all the braiding is done by $b_1$. 
This gives a short exact sequence of groups:
$$
0 \to \bZ \ast \bZ \to \PBr_3 \xrightarrow{\text{forget $b_1$}} \bZ \to 0
$$
where the kernel is identified with $\pi_1 (\bC \backslash \{ -1, 0 \} ) \simeq \bZ \ast \bZ$. In terms of braids, in the notation  of Figure \ref{fig:PBr_3-generators}, this is naturally generated by $R_2$ and $R_3$.

Putting things together, we have 
$$ K = \ker \{ \bZ\langle R_2 \rangle \ast \bZ \langle R_3 \rangle \xrightarrow{\text{forget $b_{-1}$}} \PBr_2 \}$$
A pure braid $\psi  \in \bZ\langle R_2 \rangle \ast \bZ \langle R_3 \rangle$ is in $K$ if and only if the power of $R_2$ in the abelianisation of $\bZ\langle R_2 \rangle \ast \bZ \langle R_3 \rangle$ vanishes. 
To conclude, consider the infinite cyclic cover  of the wedge of two circles with group of deck transformations $\bZ \langle R_2 \rangle$. The image of $\pi_1$ of this cover in $\bZ\langle R_2 \rangle \ast \bZ \langle R_3 \rangle$ is $K$; we now see that it is a   $\bZ^{\ast \infty}$ with generators $R_2^{-i} R_3 R_2^{i} = t_i$,  $i \in \bZ$.
\end{proof}

Mapping $R^{-i}_2 R_3 R^{i}_2$ to the Dehn twist $\tau_{S_i}$ gives a map  $\PBr_3^c \to \pi_0 \Symp_c X$; this is compatible with the description of $\tau_{S_i}$ as the monodromy of the full twist in $\gamma_i$ from \cite[Lemma 4.13]{SW}. We will now upgrade this to a representation of the whole of $\PBr_3$ as symplectomorphisms; in order to do this, we have to enlarge the group we are working in to include non-compactly supported exact symplectomorphisms.

\begin{lemma}\label{lem:PBr_3-action-ungraded}

There is a map 
\begin{equation} \label{eqn:pure-braid-monodromy}
\PBr_3 \to \pi_0 \Symp_{\ex} X
\end{equation}
extending the composition $\PBr_3^c \to \pi_0 \Symp_c X \to \pi_0 \Symp_{\ex} X$. This is compatible with the action of $\PBr_3$ on the collection of matching paths: if $\gamma$ is a matching path, $R \in \PBr_3$, and $\gamma' = R \cdot \gamma$, then the image of $R$ in $ \pi_0 \Symp_{\ex} X$ takes $S_\gamma$ to $S_{\gamma'}$ (up to Hamiltonian isotopy).

\end{lemma}

\begin{proof}

Consider $\bC^4 \times \bC^\ast \times (\bC^\ast)^2$ with coordinates $(u_1, v_1, u_2, v_2, z, a, b)$ and its standard K\"ahler form. Consider the smooth affine subvariety 
\begin{equation} \label{eqn:family-over-configuration-space}
\scrX = \{ u_1 v_1 = z-a, u_2 v_2 = z-b \}.  
\end{equation}
Projection to the final two coordinates $(a,b)$ defines a symplectic fibre bundle $\scrX \to (\bC^\ast)^2$. The fibre over $(a, b) \in (\bC^\ast)^2 \backslash \Delta$ is the space $X_{a,b}$ from Equation \ref{eq:X_{a,b}}, and the family has three-fold ordinary double points along $\Delta = \{ a=b\}$. 

Now on the one hand, $\pi_1 ((\bC^\ast)^2 \backslash \Delta)$ is naturally the pure braid group $\PBr_3$. On the other hand, given any path $\gamma: [0,1] \to (\bC^\ast)^2 \backslash \Delta$ in the smooth locus,  
we can construct parallel transport maps over $\gamma$ which are globally defined and strongly exact, by the argument of 
\cite[Lemma 2.2]{Keating-tori}.  Symplectic monodromy then defines the representation \eqref{eqn:pure-braid-monodromy}.  The fact that the fibres over points of $\Delta$ have a single ordinary double point implies that the monodromy around a path which is a  meridian to $\Delta$ (oriented anticlockwise as the boundary of a disc positively transverse to $\Delta$ at one point) is a Dehn twist (see also \cite[Lemma 4.13]{SW}).  This implies that \eqref{eqn:pure-braid-monodromy} extends the previously defined $\PBr_3^c \to \pi_0 \Symp_c X$. It is then standard, cf.~\cite{KS}, that the action is compatible with the natural action on matching paths.
\end{proof}

\begin{remark}
The proof readily gives a stronger result: we have a natural map $\PBr_3 \to \pi_0 \Symp_\partial X$, where 
 $\Symp_\partial (X)$ denotes the group of strongly exact symplectomorphisms, i.e.~diffeomorphisms $\phi$ of $X$ such that $\phi^\ast \theta = \theta + df$ for a compactly supported function $f$, where $\theta$ is the Liouville form on $X$. 
We will not need this for our purposes. 
\end{remark}

\begin{remark}
Let us justify the notation $\PBr_3^c$. Using the Morse-Bott-Lefschetz perspective, given any element of $\PBr_3$, we could try to build a symplectomorphism of $X$ by first getting fibrewise symplectomorphisms via parallel transport, and then upgrading to a symplectomorphism of the total space. (For technical reasons we use a different set-up in the present work.) 
For what elements $\psi \in \PBr_3$ does parallel transport induce the identity for fibres near infinity? Consider some arc $\tilde{\gamma}$ from $+ \infty$ to $0$, avoiding the marked points. We want to know when the monodromy induced by the concatenated path $\tilde{\gamma} \# (- \psi \tilde{\gamma})$ is the identity. Now notice that this is equivalent to asking that $\tilde{\gamma} \# (- \psi \tilde{\gamma})$
have winding number 0 about both $-1$ and $+1$, which is in turn equivalent to $\psi$ being in the kernel of the two forgetful maps to $\PBr_2 \times \PBr_2$. 
\end{remark}

Let $W^{2n}$ be a symplectic manifold with $c_1(W)=0$, and fix a compatible almost complex structure $J$. A choice of holomorphic volume form $\Theta_J \in \Gamma(\Lambda_{\bC}^n T^*W)$  defines a distinguished  fibrewise universal cover $\widetilde{\Lag}(W) \to \Lag(W)$ of the fibre bundle $\Lag(W) \to W$ whose fibre is the Lagrangian Grassmannian $Gr_{Lag}(T_wW) \cong U(n)/O(n)$. Explicitly, we have a map
\[
\Theta_J : \Lag(W) \to S^1, \qquad (v_1,\ldots,v_n) \mapsto \Theta_J(v_1\wedge \cdots \wedge v_n) / | \Theta_J(v_1\wedge \cdots \wedge v_n)|
\]
and we consider the bundle $\widetilde{\Lag}(W) \to \Lag(W)$ with fibre the pairs $(\Lambda,t) \in \Lag(W) \times \bR$ with $\Theta_J(\Lambda) = e^{i\pi t}$.

\begin{lemma} \label{lem:gradings-exist}
The map $\PBr_3 \to \pi_0\Symp_{\ex}(X)$ can be lifted to graded symplectomorphisms.

\end{lemma}

\begin{proof}
Recall the space $X_{a,b}$ from \eqref{eq:X_{a,b}} is given by $ \{u_1v_1=z-a, u_2v_2=z-b\} \subset \bC^4\times \bC^*$.  Write $f_1=u_1v_1-z+a$ and $f_2 = u_2v_2-z+b$. 
Following Griffiths, there is a holomorphic volume form on  $X_{a,b}$, sometimes symbolically written as
$$
\frac{d \log z \wedge d u_1  \wedge   d u_2 \wedge d v_1 \wedge d v_2} {d f_1 \wedge d f_2}. 
$$
We spell this out on charts. On the chart $\{ u_1 \neq 0 \neq u_2 \}$, the holomorphic volume form is given by 
$$
d \log z \wedge d u_1 \wedge d u_2  \cdot \frac{1}{ \partial f_1 / \partial v_1} \cdot \frac{1}{\partial  f_2 / \partial v_2} = d\log(z) \wedge d\log(u_1) \wedge d\log(u_2).
$$
On the charts $\{ v_1 \neq 0 \neq u_2 \}$, $\{ u_1 \neq 0 \neq v_2 \}$ and  $\{ v_1 \neq 0 \neq v_2 \}$, we get similar expressions (up to signs) by replacing $u_1$ with $v_1$, and / or $u_2$ with $v_2$.
Together these charts cover all of $X$ except the locus where  $u_1 = v_1 = 0$ or  $u_2 = v_2 = 0$.
 If $u_1 = v_1 = 0$, we must have $z=a \in \bC^\ast $ and $u_2v_2\neq 0$. Working in the chart  $\{ z \neq 0 \neq u_2 \}$, the holomorphic volume form is given by 
$$
- \frac{1}{z} \cdot \ d u_1 \wedge d u_2  \wedge dv_1  \cdot \frac{1}{ \partial f_1 / \partial z} \cdot \frac{1}{\partial  f_2 / \partial v_2} =
\frac{1}{z} \cdot \frac{1}{u_2} d u_1 \wedge d u_2  \wedge dv_1.
$$
Up to signs, it has a similar expression on the chart $\{  z \neq 0 \neq u_1 \}$. Altogether, our charts now cover $X$.

This existence of our holomorphic volume form implies that 
 the family $\scrX \to \Conf_2(\bC^*)$ over configuration space from \eqref{eqn:family-over-configuration-space} admits a relative holomorphic volume form.  Via the preceding construction, this defines a fibration $\scrX^{\infty} \to \scrX$ with fibre the universal cover $\widetilde{U(n)/O(n)}$.  A choice of Ehresmann connection for this fibration enables one to  lift (not necessarily uniquely) the parallel transport maps for $\scrX \to \Conf_2(\bC^*)$ to parallel transport maps for $\scrX^{\infty} \to \Conf_2(\bC^*)$, which in particular shows that these symplectomorphisms admit coherent gradings.
 \end{proof}

\emph{Choice of gradings for the action of $PBr_3$.} We now fix a preferred lift $\PBr_3 \to  \pi_0 \Symp_{\ex}^{\gr}(X)$, by specifying the lifts of each of the three generators $R_1, R_2$ and $R_3$.
 First, the image of $R_3$ is the Dehn twist $\tau_{0} = \tau_{S_0}$; we give it the preferred grading for a compactly supported symplectomorphism.
 Second, we grade the image of $R_2$, say $\lambda$, by asking for $\lambda^i S_0 = S_{-i}$  as graded objects. 
 Finally, to choose the grading for the image of $R_1$, say $\rho$, recall that $R_1 R_2 R_3$ generates the center of $\PBr_3$, and corresponds to the inverse Dehn twist in a boundary parallel circle; in particular, the product $\rho  \circ \lambda \circ \tau_{S_0}$ takes $S_0$ to itself as an ungraded object. Fixing a grading on $\rho$ is the same as fixing the shift of $S_0$ under $\rho \circ \lambda \circ \tau_{S_0}$; we choose this to be $[1]$.

\begin{remark}

The central map $\rho \circ \lambda \circ \tau_{S_0}$ is isotopic in $\pi_0 \Symp (X)$ to the identity (this follows from the proof of Lemma \ref{lem:PBr_3-action-ungraded}), though  through maps acting non-trivially on the boundary.
From a purely symplectic perspective, the grading of the central element $\rho \circ \lambda \circ \tau_{S_0}$ could have been chosen to be any power of a shift. (In particular, there isn't a `preferred' shift as we are not in $\pi_0 \Symp_c (X)$.) We will later see that from an HMS perspective it is natural to choose the shift $[1]$. 
\end{remark}

\subsection{$K$-theoretic actions of symplectomorphism groups}

We have that $K(\scrW (X)) = \bZ^2$ generated by two cotangent fibres $F_0$ and $F_1$, by for instance \cite[Theorem 1.13]{GPS2} or \cite{CDGG}.
This is naturally identified with $H_3(X, \partial_\infty X; \bZ)$, where $\partial_\infty X$ denotes the conical end of $X$.
Moreover, there is a pairing 
\begin{equation}\label{eq:pairing}
K(\scrF(X)) \times K(\scrW(X)) \to \bZ, \qquad (K, L) \mapsto \chi(HF^\ast (L,K)).
\end{equation}
Evaluating this on fibres and zero-sections shows that $\bZ^2  = \langle S_0, S_1 \rangle \to K(\scrF(X)) $ is injective.  
(Note we don't have an explicit description for $K(\scrF(X))$.)

\begin{definition}
Let the numerical Grothendieck group of $\scrF(X)$, $K_{num} (\scrF (X))$, be the group $K(\scrF(X)) / \ker$, where $\ker$ is the subgroup generated by compact Lagrangians $L$ s.t. $\chi(HF^\ast (L,-)) = 0$ for every object $-$ in $\cW(X)$. 
\end{definition}

The pairing \ref{eq:pairing} descends to a non-degenerate pairing $K_{num} (\scrF(X)) \times K(\scrW(X)) \to \bZ$. This implies that $K_{num}(\scrF(X)) = \bZ^2 = \langle S_0,S_1 \rangle. $ This is naturally identified with $H_3 (X; \bZ)$, and the non-degenerate pairing 
$$
K_{num} (\scrF (X)) \times K ( \cW(X)) \to \bZ
$$
is identified with the intersection pairing
$$
H_3 (X; \bZ) \times H_3 (X, \partial_\infty X; \bZ) \to \bZ.
$$

Both $\Symp^{\gr}(X) $ and $\Symp_c(X)$ (with preferred gradings) act on $K(\scrF(X))$ and $K(\scrW(X))$ 
compatibly with this pairing. This implies that there is an induced action on the numerical Grothendieck group. On the other hand, as the intersection pairing of $S_0$ with both $S_0$ and $S_1$ vanishes, and similarly for $S_1$, we see that the standard map $H_3 (X; \bZ) \to H_3(X, \partial_\infty X; \bZ)$ has zero image. 
The following is now immediate:
\begin{corollary} \label{cor:K-thy-actions}

Suppose that $f \in \pi_0 \Symp^{\gr} (X)$. Then $f_\ast$ acts on $K_{\text {num}} (\scrF)$. Moreover, $f$ induces the identity on  $K(\scrW)$ if and only if it induces the identity on $K_{\text {num}} (\scrF)$.

\end{corollary}

For arbitrary compactly supported maps, we have the following.

\begin{lemma}\label{lem:trivial_on_homology}
Compactly supported homeomorphisms act trivially on $H_3(X; \bZ)$.
\end{lemma}

\begin{proof}
Let $\phi \in \Homeo_c (X)$ be a compactly supported homeomorphism. Considering the induced action on the long exact sequence for the pair $(X, \partial_\infty X)$, we get a commutative diagram:
\begin{equation}
\xymatrix{
\ldots \ar[r] & H_3( \partial_\infty X) \ar[r] \ar[d]_{\phi_\ast} &  H_3( X) \ar[r] \ar[d]_{\phi_\ast} &  H_3( X, \partial_\infty X) \ar[r] \ar[d]_{\phi_\ast} & \ldots \\
\ldots \ar[r] & H_3( \partial_\infty X) \ar[r] &  H_3( X) \ar[r]  &  H_3( X, \partial_\infty X) \ar[r]  & \ldots
}
\end{equation}
where all coefficient rings are $\bZ$. By assumption, $\phi_\ast: H_3(\partial_\infty X) \to H_3(\partial_\infty X)$ is the identity. On the other hand, recall that in this case the map $H_3(X) \to H_3(X, \partial_\infty X)$ has zero image. It's then immediate that $\phi_\ast: H_3(X) \to H_3(X)$ is the identity. 
\end{proof}

\begin{corollary}\label{cor:aut-tor}
Any $f \in \pi_0 \Symp_c X$ acts trivially on $K_{num} (\scrF)$ and $K(\cW)$. 
\end{corollary}

\begin{proof}
This readily follows from Corollary \ref{cor:K-thy-actions} together with Lemma \ref{lem:trivial_on_homology}.
\end{proof}

Lemma \ref{lem:trivial_on_homology} will also imply that there are infinity many orbits of Lagrangians spheres in $X$; see Corollary \ref{cor:infinite-sphere-orbits} for a more general statement.
On the other hand, by the proof of Proposition \ref{prop:compact-braid-action}, we see that all of the $S_\gamma$ are contained in the orbits of the $S_i$, $i \in \bZ$, under the action of the group generated by Dehn twists in the $S_i$. 

\begin{remark}
We'll  see in Section \ref{sec:main-theorem} that the subgroup $\bZ^{\ast \infty}$ which maps to $\pi_0 \Symp_c X$ actually split injects into it. This will use mirror symmetry. If one only wants to show that the action of $\bZ^{\ast \infty}$ is faithful, one could do this instead by using a generalisation of Khovanov--Seidel's work on faithful braid group actions, see Proposition \ref{prop:generalised_Khovanov_Seidel}. 
\end{remark}

\section{The conifold resolution} \label{sec:resolution-side}

Consider the threefold ordinary double point $\{xy=(1+w)(1+t)\}\subset \bC^4$. This has a small resolution, which is the total space of  $\cO(-1)\oplus\cO(-1) \to \bP^1$.  Let $D$ be the pullback under the small resolution of the divisor $\{wt=0\} \subset \bC^4$. 
Let $Y$ be the open subset of $\cO(-1)\oplus\cO(-1)$ given by the complement of $D$. Let $C \subset \cO(-1)\oplus\cO(-1)$ denote the zero-section, which lies in the complement of the divisor $D$ and hence in $Y$.

A brief calculation shows that $\Pic(Y) \cong \bZ$. This is generated by a line bundle $\scrL$ such that $\calO_C(i) \otimes \scrL \cong \calO_C(i+1)$ (by the push-pull formula). 

Let $R = \Gamma (\calO_Y)$. We have that $\Spec R \cong  \{xy=(1+w)(1+t)\} \backslash \{wt=0\} \subset \bC^4$ is the blow-down of $Y$. 

\begin{remark} Start with $\bP^1 \times \bP^1 \times \bP^1$;  blow up $Z = \bP^1 \times \{ -1 \} \times \{ 0 \}  \cup \{ -1 \} \times \bP^1 \times \{ \infty \}$, and let $F$ be the proper transform of the toric divisor for $(\bP^1)^3$. Then one can check that the variety $Y$ is isomorphic to $\text{Bl}_Z (\bP^1)^3 \backslash Z$. In particular, $Y$ is log Calabi-Yau. 

\end{remark}

We write $D(Y)$ for the bounded derived category of $Y$.

\begin{definition} \label{defn:the-big-category} 
The category  $\scrD \subset D(Y)$ is the full subcategory of $D(Y)$ of complexes whose cohomology sheaves are (set-theoretically) supported on $C\subset Y$. 
\end{definition}

\begin{remark} \label{rem:two-definitions-of-D} Definition \ref{defn:the-big-category} is taken from \cite{HW}. The paper \cite{CPU} defines $\scrD$ as the derived category of the abelian subcategory of $\mathrm{Coh}(Y)$ of sheaves (set-theoretically) supported on $C$. This yields the same category as Definition \ref{defn:the-big-category} by  \cite[\href{https://stacks.math.columbia.edu/tag/0AEF}{Lemma 0AEF}]{stacks-project}. The alternative definition makes it manifest that $\scrD$ is determined by a formal neighbourhood of $C$ in $Y$. On the other hand, the  formal neighbourhood of a length one flopping curve is uniquely determined, and a flopping curve has normal bundle $\calO(-1)\oplus\calO(-1)$ if and only if it has length one  \cite{Reid}. So the category $\scrD$ depends only on $C$ being a $(-1,-1)$-curve. 
\end{remark}

From Remark \ref{rem:two-definitions-of-D}, or directly, one sees that  $\scrD$ is generated by $\cO_C$ and  $\cO_C(-1)$, so 
the Grothendieck group $K(\scrD)$ of $\scrD$ is $\bZ^2$, generated by $[\calO_C], [\calO_C(-1)]$.

\begin{proposition}\label{prop:auteqs-are-FM}
Suppose $\phi$ is an autoequivalence of $D(Y)$. Then $\phi$ is Fourier-Mukai, i.e.~induced by an object $P_\phi \in D(Y \times Y)$. Moreover, $P_\phi$ is unique up to isomorphism.
\end{proposition}

\begin{proof}
Existence of $P_\phi$ is in \cite[Corollary 2.16]{Lunts-Orlov}. (See also \cite[Theorem 7.13]{Ballard} and \cite[Theorem 1.1]{Canonaco-Stellari}.)  Uniqueness follows from To\"en \cite[Corollary 8.12]{Toen}, see \cite[Remark 9.11]{Lunts-Orlov}. 
\end{proof}

The contraction map $Y \to \Spec(R)$ gives $D(Y)$ the structure of an $R$-linear category, via pullback.  In particular, we can consider those autoequivalences of $D(Y)$ which are $R$-linear.

\begin{lemma}\label{lem:R-linear-fixes-points}
Let $\phi $ be an $R$-linear autoequivalence of $D(Y)$ and $x \in Y \backslash C$. Then $\phi$ preserves the skyscraper sheaf $\calO_x$ up to a shift. 
\end{lemma}

\begin{proof}
For any sheaf $\calE \in D(Y)$, one can consider the ideal
$$
\Ann (\calE) := \{ r \in R \, | \, r \cdot \calE = 0 \}.
$$
By $R$-linearity, $\Ann \calO_x = \Ann  (\phi \calO_x)$. On the other hand, note that $\Ann \calO_x$ is simply the maximal ideal of functions vanishing at $x$. Thus $\phi \calO_x$ is also supported at $x$, and hence by \cite[Lemma 4.5]{Huybrechts} agrees with $\calO_x$ up to a shift. 
\end{proof}

To simplify notation, we will denote $Y \times_{\Spec R} Y$ by $Y \times_R Y$.

\begin{lemma}\label{lem:R-linear-FM-support}
$R$-linear autoequivalences 
of $D(Y)$ preserve $\scrD$
 and are induced by Fourier-Mukai kernels with support on $Y \times_{ R} Y$. 
\end{lemma}

\begin{proof}
Let $\phi$ be an $R$-linear autoequivalence of $D(Y)$ with Fourier-Mukai kernel $P_\phi \in D(Y\times Y)$. 
By Lemma \ref{lem:R-linear-fixes-points}, for each $x \in Y \backslash C$, the fibre of $\Supp P_\phi \to Y$ above $x$ is zero dimensional.
On the other hand, the fibre of $\Supp P_\phi \to Y$ above any point is connected \cite[Lemma 6.11]{Huybrechts}. Thus for any $x \in Y \backslash C$, the fibre of  $\Supp P_\phi \to Y$ above $x$ is simply $\{ x \}$. 
By \cite[Lemma 3.29]{Huybrechts}, this means that 
$$
(\Supp P_\phi )|_{Y \backslash C \times Y \backslash C} 
= \Delta_{Y \backslash C}. 
$$
Using \cite[Lemma 6.11]{Huybrechts} again, it follows that 
$$
\Supp P_\phi \subset \Delta_Y \cup C \times C. 
$$
Both parts of the lemma follow. 
\end{proof}

Let $\Stab \scrD$ denote the space of Bridgeland stability conditions on $\scrD$ \cite{Bridgeland}. We use standard notation: we will write $(Z, \{ \cP [\theta] \}_{\theta \in \bR})$ for the stability condition with central charge $Z: K(\scrD) \to \bC$ and  full additive subcategory of phase $\theta$ given by $ \cP[\theta] \subset \scrD$.
Following \cite{Toda1}, let  $\Stab_n \scrD$ denote the subspace of normalised Bridgeland stability conditions, i.e.~stability conditions whose central charge $Z$ satisfies $Z([\calO_c])= -1$ for any $c \in C$ (note that by \cite[Lemma 3.14]{Toda1}, the class $[\cO_c]$ is independent of $c \in C$). 
 $\Stab_n \scrD$ has an open subset consisting of stability conditions with the standard heart $\scrD \cap \Coh Y$ and central charges of the form 
$$
Z_{\beta + i \zeta} (E) := - \int e^{-(\beta+ i\zeta)} \ch E \quad \text{for} \quad \beta+ i\zeta \in A(\scrD)_\bC$$
where $ A(\scrD)_\bC$ denotes the complexified ample cone. (See \cite[Lemma 4.1]{Toda1}; these correspond to the neighbourhood of the large volume limit.) Following \cite[Definition 4.2]{Toda1}, let $\Stab_n^\circ \scrD$ be the connected component of $\Stab_n \scrD$ containing this open subset. 

By  \cite[Theorem 7.1]{Toda2}, 
$\Stab \scrD$ is connected. Moreover, we have a rescaling action of $t \in \bC$ on $\Stab \scrD$:
$$
(Z, \{ \cP[\phi] \}_{\phi \in \bR}) \mapsto (e^{-i\pi t} Z, \{ \cP[\phi + \Re t] \}_{\phi \in \bR})
$$

Assembling all this together, we get a commutative diagram
$$
\xymatrix{
\bC \times \Stab_n^\circ \scrD \ar[rr]^{\simeq} && \Stab \scrD \\
2\bZ \times \Stab_n^\circ \scrD \ar[rr]^{\simeq} \ar@{^{(}->}[u] &&
\Stab_n \scrD \ar@{^{(}->}[u]
}
$$
where the top horizontal isomorphism is given by the rescaling action.

\begin{corollary}
Let $\phi$ be an $R$-linear autoequivalence of $D(Y)$ which preserves $\calO_x \in D(Y)$  for some $x\in Y\backslash C$ (with no shift). Then the action of $\phi$ on $\Stab \scrD$ preserves 
$\Stab^\circ_n \scrD$. 
\end{corollary}

\begin{proof} This follows from the preceding discussion, on noting that the rescaling action of $\bZ \subset \bC$ agrees with the action of shifts. \end{proof}

\begin{definition}
Let  
$\Auteq \scrD$ be the group of $R$-linear autoequivalences of $D(Y)$, and let $\Auteq^\circ \scrD$ be the subgroup of $R$-linear autoequivalences preserving $\calO_x \in D(Y)$ for some point $x \not\in C$. 
\end{definition}

To justify the notation, we'll see shortly that we could have equally well defined these groups in terms of $R$-linear autoequivalences of $\scrD$.

\begin{theorem}\cite{Toda1, Toda2} \label{thm:Toda-input}

(i) We have an isomorphism 
$$\Auteq \scrD \cong (\bZ\ast\bZ) \times \bZ,$$ where the first factor has generators the spherical twist in $\calO_C(-1)$ and $\otimes \scrL$, where $\scrL$ generates $\Pic(Y)$, and the final $\bZ$ factor is the grading shift. 

(ii) We have $\Auteq^\circ \scrD \cong \bZ\ast\bZ$, again with  generators the spherical twist in $\calO_C(-1)$ and $\otimes \scrL$. 
The action of $\Auteq^\circ \scrD$ on $K(\scrD)$ induces a homomorphism to the stabiliser of $[\calO_x]$ in $GL_2(\bZ) = \Aut K(\scrD)$, i.e.~the infinite dihedral group $\bZ \rtimes \bZ/2$. Let $ \Auteq^\circ_{\Tor} \scrD$ be its kernel. Then the image of this homomorphism is $\bZ$, and we then obtain a diagram with exact rows
\begin{equation}\label{eq:auteq_commutative_diagram}
\xymatrix{
0 \ar[r] &  \Auteq^\circ_{\Tor} \scrD \ar[r] & \Auteq^\circ \scrD \ar[r] & \bZ \ar[r] & 0 \\
0  \ar[r]  & \bZ^{\ast \infty} \ar[r] \ar[u]^{\simeq} & \bZ \ast \bZ  \ar[r] \ar[u]^{\simeq} &  \bZ  \ar[r] \ar@{=}[u] &  0
}
\end{equation}
The vertical isomorphism maps the standard generators of $\bZ^{\ast\infty}$ to $T_{\calO_C(i)}$, $i \in \bZ$; and the standard generators of $\bZ \ast \bZ$ to $T_{\calO_C(-1)}$ and $\otimes \scrL$. 

\end{theorem}

\begin{proof}
For (i), this is \cite[Theorem 7.7]{Toda2}. For (ii), the description of $ \Auteq^\circ \scrD$ is clear. To conclude, consider $\scrL$ a line bundle generating $\Pic (Y) \cong \bZ$. Up to possibly replacing $\scrL$ with its dual line bundle, its action by tensor on $K (\scrD)$, with respect to the basis $\{ [\cO_C] - [\cO_C(-1)], [\cO_C] \} $, is given by the matrix $\begin{pmatrix} 1 & 1 \\ 0 & 1 \end{pmatrix}$. As $\otimes \scrL^{i} \circ T_{\calO_C} \circ \otimes \scrL^{-i} =  T_{\calO_C(i)}$, the rest of the claim is then immediate. 
\end{proof}

The proof shows that the $\bZ$ quotient on the top line of equation \ref{eq:auteq_commutative_diagram} is naturally identified with $\Pic(Y)$.

\begin{remark}\label{rmk:defs-of-AuteqD}
For instance comparing with \cite[Theorem 1.5]{HW}, we see that $\Auteq^\circ_{\Tor} \scrD$ could equivalently have been defined as $R$-linear autoequivalences \emph{of the subcategory $\scrD$} preserving $\Stab^\circ_n \scrD$. 
\end{remark}

We'll use the following immediate corollary.

\begin{corollary}\label{cor:trivial-K-theory}
Suppose that $\phi \in \Auteq \scrD$ induces the identity on $K(\scrD)$. Then $\phi$ must lie in the subgroup $\bZ^{\ast \infty} \times 2\bZ$, where $\bZ^{\ast \infty} \subset \bZ \ast \bZ$ is the subgroup generated by spherical twists in the $\calO_C(i)$; and $2\bZ$ denotes even shifts. 
\end{corollary}

\section{Proof of Theorem \ref{thm:main}}\label{sec:main-theorem}

For an $A_{\infty}$-category $\scrC$ we write $D(\scrC) := H^0(\Tw(\scrC))$ for the cohomological category of the category of twisted complexes; this is a triangulated category in the classical sense.  We remark that the group $\Auteq \scrC$ of $A_{\infty}$-autoequivalences of $\scrC$ admits a natural homomorphism to the group $\Auteq(D(\scrC))$ of triangulated equivalences of $D(\scrC)$.

The primary mirror symmetric  input to the proof of Theorem \ref{thm:main} is the following:

\begin{theorem} \cite[Theorem 4.2]{CPU}  \label{thm:hms-equivalence}
There is a quasi-equivalence of $\bC$-linear triangulated categories
\begin{equation} \label{eqn:hms-equivalence}
\Upsilon: D \cW(X) \stackrel{\sim}{\longrightarrow} D(Y)
\end{equation}
 with $\Upsilon(S_i) = \calO_C(-i)$.  
  \end{theorem}
 
It follows that $D(\scrW)$ is in fact split-closed. 
 By construction, the mirror equivalence identifies the subcategory $\langle S_0, S_1 \rangle$ of $D\cW(X)$ generated by $S_0$ and $S_1$ with $\scrD$;  passing to $K$-theory, the numerical Grothendieck group $K_\text{num} \scrF(X)$ is identified with $K(\scrD)$.  Moreover, \cite{CPU} explicitly identify the objects mirror to $\calO_Y$ and $\calO_Y(1)$: these are Lagrangian $\bR^3$s, say $L_0$ and $L_1$, fibred over the base of the Morse-Bott-Lefschetz fibration, see Figure \ref{fig:spherical-conventions}, cf.~\cite[Figure 1.2 and Theorem 1.2]{CPU}.

 \begin{corollary}
 The equivalence $\Upsilon$ of \eqref{eqn:hms-equivalence} entwines the $SH^0(X)$-linear structure of $D\cW(X)$ and the $R$-linear structure of $D(Y)$.
 \end{corollary}

\begin{proof}
The existence of an isomorphism of $\bC$-algebras $SH^0(X) \cong \Gamma(\calO_Y)=R$ is proved in \cite[Appendix]{CPU}. 
Both linear structures can be viewed as the general fact that a $k$-linear $dg$-category or $A_{\infty}$-category $\scrC$  is also linear over the Hochschild cohomology $HH^0(\scrC)$. At the cohomological level (which is all that we require) this is classical, and a chain-level version is established in   \cite[Section 5.1]{Pomerleano_intrinsic}. By \cite{GPS1}, the wrapped category $\cW(X)$ is generated by $L_0$ and $L_1$ and is homologically smooth and non-degenerate. This means in particular that the open-closed map $HH_*(\cW(X)) \to SH^*(X)$ is an isomorphism. Since $\cW(X)$  is also a Calabi-Yau category \cite{Ganatra-thesis}, the closed-open map $SH^*(X) \to HH^*(\cW(X))$ is also an isomorphism. The action of $\Upsilon$ on $HH^*$ therefore induces a distinguished isomorphism $SH^0(X) \cong R$.  The result follows. 
\end{proof}

Recall the exact, Maslov zero torus $T$ introduced in Section \ref{Subsec:exact-lagrangian-submanifolds}, which was graded in Section \ref{Subsec:gradings}. 

\begin{lemma}\label{lem:mirror-to-point}
Let $(T, \zeta) \in \Ob \cW(X)$ be any Lagrangian brane associated to $T$. Then under the mirror symmetry isomorphisms above, $(T, \zeta)$ is mirror to $\calO_x[k]$ for some $x \in Y \backslash C$ and shift $k \in \bZ$. 
\end{lemma}

\begin{proof}
Let $\calE$ be the mirror to $(T, \zeta)$. 
By inspection, we can find Hamiltonian isotopies so that the image of $T$ is disjoint from the sphere $S_0$, respectively $S_1$; and the Floer theory of $T$ with each of $L_0$ and $L_1$ has a single generator, see Figure \ref{fig:spherical-conventions}.
By considering Floer cohomology with $S_0$ and $S_1$,  $\calE$ is orthogonal to $\calO_C$ and $\calO_C(-1)$, and hence the whole of $\scrD$. 
Since $\calE$ is orthogonal to $\calO_p$ for any $p \in C$, $\calE$ has support on $Y \backslash C$. 

\cite[Theorem 5.3]{CPU} shows that $\calO_Y \oplus \calO_Y(1)$ is a tilting object in $D(Y)$. The torus $(T, \zeta)$ is compact, and so has finite rank Floer cohomology with $L_0$ and $L_1$. This implies that its mirror $\calE$ has finite rank total $\Ext$ groups with the tilting object. On the other hand, it is shown in the proof of \cite[Theorem 7.2.1]{BCZ} that whenever a smooth variety $Z \to A$ is projective over $A$, where $A$ is affine, and $\calE_Z$ is a tilting object for $Z$, and $F \in D(Z)$ has non-proper support, then $\Ext^\ell(\calE_Z, F)$ has infinite rank for some $\ell$. 
This implies that in our case $\calE$ must have proper support. 

Finally, the only proper subvarieties of $Y \backslash C$ are finite unions of points. Since $(T, \zeta)$ has the self-Floer cohomology of a three-torus, the mirror $\calE$ is simple and has no negative self-Exts. By \cite[Lemma 4.5]{Huybrechts}, $\calE = \calO_x[k]$, some $k \in \bZ$. 
\end{proof}

\begin{remark}
Varying the local system $\zeta$ on $T$, one expects the family of branes $(T, \zeta)$ in $X$ to correspond  to a cluster chart $(\bC^\ast)^3$ on the mirror $Y$ (suitably interpreted,  to toric SYZ fibres). 
There are other exact Lagrangian tori in $X$ with vanishing Maslov class, for instance fibred over an $S^1$ in the base $\bC^*$ of the Morse-Bott-Lefschetz fibration which encloses one critical fibre and the puncture. Appropriate such tori do have non-trivial Floer theory with $S_0$ or $S_1$, for certain choices of local systems. In that sense Lemma \ref{lem:mirror-to-point} is saying that among all fibred exact tori, our choice of $T$ is the ``correct'' one to get a cluster chart which avoids (the structure sheaves of) all points of $C$. Compare to the closely related two-dimensional situation studied in \cite{Auroux:Gokova}.
\end{remark}

Lemma \ref{lem:R-linear-fixes-points} then implies that for any $\phi \in \pi_0 \Symp^{\gr} X$, $\phi (T, \zeta)$ is equal to $(T,\zeta)$ up to a shift  in $D \cW(X)$.

\begin{lemma}

 For any $\phi \in \pi_0 \Symp_c X$, equipped with the preferred grading, $\phi (T, \zeta)$ is equal to $(T,\zeta)$ in $D \cW(X)$.

\end{lemma}

\begin{proof}
Recall from Lemma \ref{lem:compact-support-implies-graded} that any compactly supported symplectomorphism of $X$ is gradeable and canonically graded. Such a map therefore acts on the set of graded Lagrangian submanifolds in $X$. The action of a symplectomorphism on the set of objects $\mathrm{Ob}\,\cW(X)$ of the wrapped category is compatible with the action on graded Lagrangians.  

 Let $z$ be the complex coordinate on $X$ pulled back from the $\bC^\ast$ factor of $\bC^2 \times \bC^2 \times \bC^\ast$ and let $\alpha$ be the angle coordinate for $z$.  
For all $t \in \bR_{\geq 0}$, let $\sigma_t$ be the time $t$ flow of the symplectic vector field $Z_\alpha$ which is $\omega$-dual to the closed one-form $d \alpha$. (Our convention is that this is inward-pointing with respect to the base $\bC^\ast$.) We consider the path $\{ \sigma_t \circ \phi \circ \sigma_t^{-1}\}_{t \in \bR}$ of graded symplectomorphisms. For sufficiently large $t$, the map $\phi' := \sigma_t \circ \phi \circ \sigma_t^{-1}$ has support in some compact set $V$ disjoint from $T$; moreover, we can assume that $X \backslash V$ is path connected (and so contains both $T$ and a conical end of $X$). As we are using the canonical grading for the map $\phi'$, 
we see that it must therefore fix the grading on $(T,\zeta)$.  
On the other hand, recall that we already know  that $\phi$ fixes $(T,\zeta)$ up to shift. The action of compactly supported symplectomorphisms on the set of graded Lagrangians only depends on the path-component in the group $\Symp_c(X)$ with its $C^{\infty}$-topology.  Since $\phi$ lies in the same component of the space of graded compactly supported symplectomorphisms as $\phi'$, it must also fix the grading of $(T,\zeta)$.
\end{proof}

\begin{definition}
Define $\Symp^\circ (X) \subset \Symp^{\gr} (X)$ to be the subgroup of graded symplectomorphisms which take $(T, \zeta)$ to itself in $D\cW(X)$. 
\end{definition}

\begin{lemma} \label{lem:symplectic-image-PBr_3-revisited}
Consider the map $\PBr_3 \to \pi_0 \Symp^{\gr} X$ defined earlier. Then the intersection of the image with $\pi_0 \Symp^\circ  X$ is precisely the subgroup generated by $\tau_{S_0}$ and $\lambda$, i.e.~the image of the subgroup $\bZ \ast \bZ$ generated by $R_2$ and $R_3$. 
\end{lemma}

\begin{proof}
By construction,   the image of $R_3$ is  $\tau_{S_0}$, with its preferred grading as a compactly supported map; and the image of $R_2$ is $\lambda$, a non-compactly supported map which maps $S_i$ to $S_{i+1}$ for all $i \in \bZ$. 
Now $\lambda$ induces an $R$-linear autoequivalence of $\cW(X)$, and the mirror autoequivalence in $\Auteq \scrD$ must take $\calO_C(-i)$ to $\calO_C(-i-1)$. By Theorem \ref{thm:Toda-input}, the only possibility is for this map to be $\otimes \scrL$, which fixes $\calO_x$. By Lemma \ref{lem:mirror-to-point}, we then have $\lambda \in \pi_0 \Symp^\circ X$. 

Finally, recall that we choose gradings on the image of $\PBr_3$ so that the generator of the center, i.e.~$R_1 R_2 R_3$, maps to a shift by one. The claim then follows.
\end{proof}

Any element of $\pi_0 \Symp^{\gr} (X)$ induces an $SH^0(X)$-linear autoequivalence of $\cW(X)$.
Lemma \ref{lem:R-linear-FM-support}, together with the HMS equivalence $\Auteq_{SH^0}(\scrW (X)) \simeq \Auteq_R(D(Y)) = \Auteq \scrD$, then implies that the image of 
$\pi_0\Symp^{\gr}(X)$ in $ \Auteq_{SH^0}(\scrW(X))$
preserves the subcategory $\langle S_0,S_1 \rangle$ (equivalent under mirror symmetry to $\scrD$). Moreover, from our definitions and Corollary \ref{cor:aut-tor}, any element of $\pi_0\Symp^\circ(X)$ lands not only in $ \Auteq(\scrD) $ but in fact in the subgroup  $\Auteq^{\circ}(\scrD)$. Altogether, this gives a diagram:

\begin{equation}
\xymatrix{
\pi_0\Symp_c (X) \ar[r] \ar[d] & \pi_0\Symp^{\circ} (X)  \ar[d]\\
\Auteq_{\Tor}( \langle S_0,S_1 \rangle)  \ar[r] \ar[d]_{\simeq} & \Auteq(\langle S_0, S_1 \rangle) \ar[d]_{\simeq} \\
\Auteq^{\circ}_{\Tor}(\scrD) \ar[r] & \Auteq^{\circ}(\scrD)
}
\end{equation}
with the bottom vertical arrows the HMS equivalences.  

Recall from Lemma \ref{cor:K-thy-actions} that $\pi_0 \Symp^{\gr}(X)$, and in particular $\pi_0 \Symp^\circ(X)$, acts on the numerical $K$-theory $K_{\text{num}} (\scrF (X))$. 

Passing to the mirror side, consider the image of $\pi_0 \Symp^{\gr}(X)$ in $\Auteq^\circ \scrD$. As the image of $\Auteq^\circ \scrD$ in $GL(K(\scrD))$ lies in a $\bZ$ subgroup, it follows that the image of $\pi_0 \Symp^{\gr}(X)$ in $GL(K_\text{num} \scrF(X))$ lies in the mirror $\bZ$ subgroup. We are now ready to state a categorical version of our main theorem.

\begin{theorem}\label{thm:main_categorical}

(i) For graded symplectomorphisms, we have
\begin{equation}
\PBr_3  \hookrightarrow \pi_0 \Symp^{\gr} (X) \twoheadrightarrow \Auteq_{SH^0} D\cW(X) \cong \Auteq \scrD \cong \PBr_3,
\end{equation}
where under the decomposition 
$\PBr_3 \cong (\bZ   \ast \bZ  ) \times \bZ  $ with generators $R_2$, $R_3$ and $R_1 R_2R_3$,
$R_2$ maps to $\lambda$, $R_3$ maps to $\tau_{S_0}$, and $R_1 R_3 R_3$ maps to the shift. Moreover, these group homomorphisms compose to give the identity.

(ii) Exact symplectomorphisms which preserve $(T, \zeta) \in \Ob D \cW(X)$ fit into the following commutative diagram:
\begin{equation}
\xymatrix{
 \bZ^{\ast \infty} \ar[r] \ar@{^{(}->}[d]& \bZ\ast\bZ  \ar[r] \ar@{^{(}->}[d] &  \bZ  \ar@{=}[d]  \\
  \pi_0 \Symp_c (X) \ar[r] \ar@{->>}[d] & \pi_0 \Symp^\circ (X) \ar@{->>}[d] \ar[r] & \bZ  \ar@{=}[d]  \\
  \Auteq^\circ_\text{Tor} \scrD \ar[d]^{\simeq} \ar[r] & \Auteq^\circ \scrD \ar[r] \ar[d]^{\simeq} & \bZ \ar@{=}[d]  \\
 \bZ^{\ast \infty} \ar[r] & \bZ\ast\bZ  \ar[r]  &  \bZ  
}
\end{equation}
where the vertical compositions from the top to bottom lines are the identity. The map $\pi_0 \Symp^\circ (X) \to \bZ$ is given by taking the action on $K_\text{num} (\scrF(X))$, itself naturally identified with $K(\scrD)$, and noting that the image lies in the $\bZ$ subgroup.

\end{theorem}

\begin{proof}
Part (i) is immediate from Theorem \ref{thm:Toda-input} and the preceeding construction and discussion of the pure braid group action $\PBr_3 \to \pi_0 \Symp^{\gr} X$. 

For part (ii), the map from the first to the second line is defined by combining the Section \ref{sec:symplectic-side} constructions with Lemma \ref{lem:symplectic-image-PBr_3-revisited}. 
Lemma \ref{lem:trivial_on_homology} implies that the image of $\pi_0 \Symp_c(X)$ lies in the subgroup $\Auteq^\circ_{\Tor} \scrD$ of $\Auteq^\circ \scrD$, which allows us to go from the second line to the third. 
The isomorphism of the final two rows is given by Theorem \ref{thm:Toda-input}. 

To check that composition from the first to the final row is an isomorphism, we proceed as follows. For the first column, this follows from having identified the twists $\tau_{S_i}$ and $T_{\calO_C(-i)}$, $i \in \bZ$ in the homological mirror symmetry isomorphism of Theorem \ref{thm:hms-equivalence}. 

For the second column, recall we have a preferred non-compactly supported symplectomorphism $\lambda$, together with a grading choice such that it lives in $\pi_0 \Symp^\circ (X)$. With that choice of grading, 
$\lambda$ takes $S_i$ to $S_{i+1}$, for all $i \in \bZ$ (with no shifts). 
At the level of $K$-theory, this is the same action as the tensor with our fixed generator $\scrL$ for $\Pic(Y)$.
This implies that $\pi_0 \Symp^\circ (X)$ surjects onto $\bZ$. 
Surjectivity on the second column then follows from the  exactness of the sequence
$$
 \Auteq^\circ_\text{Tor} \scrD \to \Auteq^\circ \scrD \to \bZ.
$$

Given that the composition from the first to the final row is an isomorphism, injectivity from the first to the second row, and surjectivity from the second to the third row, are then automatic. 
\end{proof}

 Theorem \ref{thm:main} immediately follows. 

\section{Classification of sphericals}\label{sec:sphericals}

Let $\calS$ be the collection of Lagrangian spheres which we know of in $X$: the images of the $S_i$, $i \in \bZ$, under the action of the group $\bZ^{\ast \infty}$ generated by the Dehn twists in the $S_i$. In other words, these are all the spheres associated to matching paths between the two critical values in the base of our Morse-Bott-Lefschetz fibration, with all possible choices of gradings.

\begin{definition}\label{def:spherical}\cite[Section (3a)]{Seidel-Am-Milnor}
We say an object $S$ of $D\cW(X)$ is spherical if 
$\HF^\ast (S,S) \cong H^\ast (S^3; \bZ)$, and   for any object $Z$ of $D \cW(X)$, 
$\HF^\ast(S, Z)$  and $\HF^\ast(Z,S)$ are of finite total rank
and the composition
$$
\HF^{3-\ast} (Z,S) \otimes 
\HF^\ast (S,Z)   \to \HF^3(S,S) \cong \bZ
$$
is a non-degenerate pairing. 
\end{definition}

\begin{theorem}\label{thm:classification_sphericals}
Let $S$ be a spherical object in $D\cW(X)$. Then $S$ is quasi-isomorphic to an element of $\calS$ (as a Lagrangian brane, i.e.~together with a choice of grading).
\end{theorem}

We will break the proof into a number of stages.

\subsection{Twists act trivially on numerical $K$-theory}

Fix a spherical object $S \in D\cW(X)$ and consider the associated spherical twist $T_S$.  By \cite[Lemma 5.46]{Pomerleano_intrinsic} we can assume that the category $D\scrW(X)$ of perfect modules over $\scrW(X)$ has a strictly $SH^0(X)$-linear chain model. The total evaluation morphism $\Hom(S,\bullet)\otimes S \to \bullet$ is then $SH^0$-linear, and the cone  inherits an $SH^0(X)$-linear structure. This means that $T_S$ is linear as an autoequivalence.  By Theorem \ref{thm:main_categorical}, we know that $T_S$ is an element of $\Auteq_{SH^0} D\cW(X) \cong \Auteq \scrD \cong \PBr_3$. 

\begin{lemma} \label{lem:sphericals-act-trivially}
$T_S$ acts trivially on $K(\scrD) \cong K_{num}(\scrF(X))$.  
\end{lemma}

\begin{proof}
Under the equivalence $\Upsilon: D\cW(X) \to D(Y)$ of Theorem \ref{thm:hms-equivalence}, $S$ corresponds to some complex $\calE_S$ of sheaves on $Y$. As in the proof of Lemma \ref{lem:mirror-to-point},  the argument in the proof of \cite[Theorem 7.2.1]{BCZ} shows that $\calE_S$ must have proper support in $Y$ (in the sense of \cite[Definition 3.8]{Huybrechts}, meaning the support is the union of the supports of the cohomology sheaves): otherwise it would have infinite rank morphisms in some degree with the tilting object in $D(Y)$, which violates the definition of being spherical.   As $\calE_S$ is also indecomposable, it must have connected support, which by the classification of proper subvarieties of $Y$ implies that the support is either $C$ or a single point. By \cite[Lemma 4.5]{Huybrechts}, the support must be $C$, and so in fact $\calE_S$ belongs to the category $ \scrD \subset D(Y)$.  It follows that $S$ belongs to the subcategory of $\cW(X)$ generated by $S_0$ and $S_1$ and in particular belongs to the compact Fukaya category $\scrF(X) \subset \cW(X)$.

The spherical twist $T_S$ acting on $\scrF(X)$  fits into an exact triangle 
\begin{equation} \label{eqn:twist-triangle}
\xymatrix{
\Hom_{\scrF}(S,\bullet)\otimes S \ar[r] &  \id \ar[r] &  T_S \ar@/^1pc/[ll]
}
\end{equation}
and hence acts on (numerical) $K$-theory by the Picard-Lefschetz type transformation
\[
A \mapsto A - \chi(S,A)[S]
\]
where $\chi(S,A) \in \bZ$ is the Euler characteristic of $\Hom_{\scrF}(S,A)$.  Since the intersection pairing on $K_{num}(\scrF(X)) = H_3(X)$ vanishes, $\chi$ vanishes identically.
\end{proof}

Lemma \ref{lem:sphericals-act-trivially} combined with 
Corollary \ref{cor:trivial-K-theory} then implies that $T_S$ lies in $\bZ^{\ast \infty} \times 2\bZ \subset (\bZ \ast \bZ) \times \bZ \cong \PBr_3$. 

\begin{remark} \label{rmk:coeff_fields}
Since we have always defined the category $\scrD$  over $\bC$, we have implicitly begun with a spherical object $S \in D\scrW(X;\bC)$ above.
On the other hand, the proof of Theorem 1.1 shows that the twist $T_S \in \Auteq_{SH^0}(D\scrW(X;\bC)) = PBr_3$ lies in a subgroup that we have identified as coming from $\pi_0\Symp_c(X)$. As $\pi_0\Symp_c(X)$ acts on $D\scrW(X;\bZ)$, we see that $T_S$ must also actually be defined over $\bZ$.  In particular, it makes sense below to consider $HF(L, T_S^k(L’);\bZ/2)$ for any objects $L, L’ \in \scrW(X;\bZ)$.
\end{remark}

\subsection{A localisation argument}\label{Subsec:localisation}

The key input  for Theorem \ref{thm:classification_sphericals} is the following analogue, in our setting, of a well-known theorem \cite{KS} due to Khovanov and Seidel, which localises the computation of Floer cohomology to the fixed locus of a finite group action.  

For matching paths $\gamma_a, \gamma_b$ in the base $\bC^*$ of the Morse-Bott-Lefschetz fibration $\pi: X\to \bC^*$, we let $I(\gamma_a, \gamma_b)$ denote the minimal (unsigned, i.e.~geometric) interior intersection number between $\gamma_a$ and $\gamma_b$, relative to the their end points. Here isotopies of paths are not allowed to cross the puncture; and `interior' means that we're not including the end points in our count.

\begin{proposition}\label{prop:generalised_Khovanov_Seidel}
Suppose $S_a, S_b \in \calS$ correspond to matching paths $\gamma_a, \gamma_b \subset \bC^*$.  
Then the Floer homology of $S_a$ and $S_b$ with coefficients in $\bZ/2$ satisfies
$$
\dim \HF(S_a, S_b;\bZ/2) = 4 + 4 I(\gamma_a, \gamma_b)
$$
\end{proposition}

The proof of Proposition \ref{prop:generalised_Khovanov_Seidel} has two ingredients: the existence of a $\bZ/2 \times \bZ/2$ action on $X$, preserving the Lagrangians, and with the property that there are no holomorphic discs in the fixed point locus; and an equivariant transversality argument which allows us to compare Floer theory in $X$ with Floer theory in that fixed point locus. We describe the involutions and the reduction to existence of an equivariantly transversal $J$ in this section; the existence of such a $J$ will be broken down into Sections \ref{Subsec:equivariant-transversality} through \ref{sec:stable-normal-trivialisations}. 

Recall $X$ is given explicitly by $$X = \{ u_1 v_1 = z-1, \, u_2 v_2 = z+1 \} \subset \bC^4 \times \bC^\ast.$$ 

$X$ carries two commuting involutions: $\iota_1$, trading $u_1$ and $v_1$; and 
$\iota_2$, trading $u_2$ and $v_2$.  Let $G = \bZ/2\times\bZ/2$ be the group generated by the $\iota_j$.

The common fixed point locus of both $\iota_1$ and $\iota_2$, say $X^{\inv}$, has equation $\{ p^2 = z-1, \, q^2 = z+1 \} \subset \bC^2 \times \bC^{\ast}$. This is smooth, and Weinstein with the restriction of the form from $X$. (Topologically, $X^{\inv}$ is a sphere with six punctures.)

Projection $X^{\inv} \to \bC^*$ to the $z$ co-ordinate, i.e.~to the base of the Morse-Bott-Lefschetz fibration, defines 
a four-fold cover of $\bC^\ast$, branched over $\pm1$. 
Note that the branching pattern is as follows: say the sheets are labelled by $(\pm p, \pm q)$. Then above $+1$, the sheets with the same value of $q$ meet; and above $-1$, the sheets with the same value of $p$ meet.

For any matching path $\gamma$, $S_\gamma$ meets $X^{\inv}$ cleanly in an embedded $S^1$, say $s_\gamma$ (which is itself a matching sphere in $X^{\inv}$). The circle $s_{\gamma}$  is a 4-fold cover of $\gamma$ under projection to $z$.

\begin{lemma} \label{lem:intersection-points-invariant}
After a Hamiltonian isotopy of $X$, the matching spheres $S_a$ and $S_b$ meet exactly in the finite set of transverse intersection points of the curves $s_a$ and $s_b$. In particular, all intersection points lie in the common fixed locus of the involutions $\iota_1$ and $\iota_2$. 
\end{lemma}

\begin{proof}
Assume that two matching paths $\gamma_a$ and $\gamma_b$ intersect transversally in $k$ points apart from the end points. Then by construction, the
 $S_a$ and $S_b$ intersect cleanly in two $S^1$s (above each end point) and $k$ copies of $T^2$ (above each interior intersection point); and their restrictions to $X^{\inv}$, i.e.~the 
 matching spheres $s_a$ and $s_b$ in $X^{\inv}$, intersect transversally in $4+4k$ points (with the end points contributing two each, and each `interior' intersection point contributing 4).  
  
There is a symplectic neighbourhood theorem for cleanly intersecting Lagrangian submanifolds, saying that near a component $B$ of the clean intersection locus they are modelled by the conormal bundle of $B$ inside the cotangent bundle of either one. Starting from this, and following the proof of \cite[Proposition 5.15]{KS}  (which treats the case of a single involution),  there is an $(\iota_1, \iota_2)$-equivariant small Hamiltonian perturbation of $S_a$, fixing the invariant locus, so that $S_a$ and $S_b$ intersect transversally in exactly the intersection points of $s_a \pitchfork s_b$. \end{proof}

Following \cite{FLP, KS}, we say that matching paths $\gamma_a$ and $\gamma_b$ intersect \emph{minimally} if they intersect transversally, and the following holds: assume we are given any two points $z_{-} \neq z_{+}$ in $\gamma_a \cap \gamma_b$ which are not both endpoints of the matching paths, and two arcs $\alpha \subset \gamma_a$, $\beta \subset \gamma_b$ ending on $z_{-}, z_{+}$  such that those arcs have no other intersection points. Then if $K$ is the connected component of $\bC \backslash (\gamma_a \cup \gamma_b)$ bounded by $\alpha \cup \beta$, and $K$ is topologically an open disc, then it must contain the puncture $\{ 0 \}$.

Using \cite[Expos\'e III]{FLP}, we can assume that after isotopy rel end points, the matching paths $\gamma_a$ and $\gamma_b$ intersect minimally, in $I(\gamma_a, \gamma_b)$ points. Any isotopy of a matching path $\gamma$ in the base (rel endpoints) is covered by a Hamiltonian isotopy of $S_\gamma$ in  $X$ and a Hamiltonian isotopy of $s_\gamma$ in $X^{\inv}$. Thus after Hamiltonian isotopy in $X$, we may assume $S_a$ and $S_b$ intersect transversally in $4+4I(\gamma_a, \gamma_b)$ points, all fixed by both $\iota_j$.

\begin{lemma} \label{lem:condition-S2-satisfied}
There is no continuous map $u: [0,1]^2 \to X^{\inv}$ with the following properties: $u (0,t) = x_{-}$, $u(1,t) = x_{+}$ for all $t$, where $x_{-} \neq x_{+}$ are points of $S_a \cap S_b$, and $u(s,0) \in S_a, \,  u(s,1) \in S_b$ for all $s$. 
\end{lemma}

In the terminology of Khovanov--Seidel, this is saying that $S_a$ and $S_b$ satisfy Condition (S2) of \cite[Lemma 5.14]{KS}: there are no bigons between $S_a$ and $S_b$. 

\begin{proof} 
By \cite[Lemma 3.1]{KS} (see also \cite[Expos\'e III]{FLP}), this is equivalent to saying that $S_a$ and $S_b$ intersect minimally on $X^{\inv}$. In other words, the claim is an analogue in our setting of \cite[Lemma 3.8]{KS}. We follow their proof closely. Using the same notation, let $U \subset \bC$ be a thickening of $\gamma_a \cup \gamma_b$ such that $\bC \backslash U$ is a surface with corners, and let $K$ be a compact connected component of $\bC \backslash U$. Then we have exactly the same case analysis as in the proof of \cite[Lemma 3.8]{KS}: at least one of the following must be true:
\begin{enumerate}
\item $K$ is not contractible;

\item $K$ has at least three corners;

\item $K$ has two corners and contains $\{ 0 \}$. 

\end{enumerate}
If $(1)$ holds, then $\pi^{-1}(K) \cap X^{\inv}$ may have multiple connected components, but none of them can be contractible, considering fundamental groups of covers. If $(2)$ holds, each connected component of $\pi^{-1}(K) \cap X^{\inv}$ has at least three corners (and so is not a bigon). Finally, if $(3)$ holds, as $K$ contains the puncture $\{ 0 \} \subset \bC$, $\pi^{-1}(K) \cap X^{\inv}$ must contain the preimages of the puncture, and so again is not a bigon. All in all, this shows that no connected component of $X^{\inv} \backslash (\pi^{-1} (U) \cap X^{\inv})$ can be a bigon. 
\end{proof}

\begin{lemma} \label{lem:equivariant-transversality-suffices}
Suppose that there is a compatible $J = \{J_t\}_{t \in [0,1]}$ on $X$ with the properties that 
\begin{enumerate}
\item $J$ is $G$-equivariant;
\item the moduli spaces of $J$-holomorphic curves $u: \bR\times [0,1] \to X$ with boundary conditions on $S_a$ and $S_b$ are regular. 
\end{enumerate}
Then 
\[
\HF(S_a,S_b;\bZ/2) \cong \HF(s_a,s_b;\bZ/2)
\]
\end{lemma}

\begin{proof}
We have arranged that all intersection points lie in the fixed locus, so the underlying Floer cochain groups agree. 

Now suppose that $u: D \to X$ is a Floer holomorphic disc contributing to the differential and that $\Im (u) \nsubseteq X^{\inv}$. Then notice that each of $u, \iota_1 \circ u, \iota_2 \circ u$ and $ \iota_1 \circ \iota_2 \circ u$ give holomorphic discs; moreover, these are either a pair of distinct holomorphic discs (e.g.~if $\Im (u) \subset \Fix (\iota_1)$ but $\Im (u) \nsubseteq \Fix (\iota_2)$) or a quadruple (if $\Im (u) \nsubseteq \Fix (\iota_1)$ and $\Im (u) \nsubseteq \Fix (\iota_2)$). Equivariant transversality means that all these discs are regular, so as we are working with $\bZ/2$ coefficients, the contribution to the differential is zero.  

It follows that the differential on $\CF(S_a,S_b;\bZ/2)$ co-incides with that on $\CF(s_a,s_b;\bZ/2)$.  
\end{proof}

For computing Floer cohomology of curves on a surface, one can achieve regularity with a time-independent almost complex structure. The computation of $\HF(s_a,s_b;\bZ/2)$ is then 
completed by the following immediate consequence of Lemma \ref{lem:condition-S2-satisfied}. 

\begin{lemma}\label{lem:no-invariant-discs}
There are no non-constant holomorphic discs with boundary on $s_a$ and  $s_b$ contained wholly in the fixed locus $X^{\inv}$. 
\end{lemma}

Lemma \ref{lem:no-invariant-discs} in turn readily implies that for our choices,
\[
\HF(s_a,s_b;\bZ/2) \cong \CF(s_a,s_b;\bZ/2)
\]
which has rank $4+4I(\gamma_a, \gamma_b)$.  Subject to achieving the hypotheses of Lemma \ref{lem:equivariant-transversality-suffices}, this completes the proof of Proposition \ref{prop:generalised_Khovanov_Seidel}.

\subsection{Equivariant transversality}\label{Subsec:equivariant-transversality}

Let $(M,\omega)$ be an exact symplectic manifold with a symplectic involution $\iota: M \to M$ which setwise preserves a pair of exact Lagrangian submanifolds $L_0, L_1 \subset M$. We write $M^{\iota}$ for the fixed locus, and $L_j^{\iota}$ for the fixed set of $\iota|_{L_j}$, which is Lagrangian in $M^{\iota}$. We suppose that the $L_i$ intersect transversally, which implies that the $L_i^{\iota}$ intersect transversally  in the fixed point locus. Let $n=\dim_{\bC}(M)$, $n_{\iota} = \dim_{\bC}(M^{\iota})$ and let $n_{\anti} = n - n_{\iota}$ be the complex codimension of the fixed locus. The normal bundle $\nu_{M^{\iota}} \to M^{\iota}$ is a complex bundle of rank $n_{\anti}$;   this is the $(-1)$-eigenspace of the action of $\iota$ on $T_pM$ for $p\in M^{\iota}$.

We consider $\iota$-equivariant paths of compatible almost complex structures $J = \{J_t\}_{t\in [0,1]}$ on $M$. Khovanov and Seidel \cite[Proposition 5.13]{KS} explain how to choose $J_t$ so as to achieve regularity for those curves contributing to $\CF(L_0,L_1)$ which are not contained wholly in the fixed locus $M^{\iota}$ of the involution.  

Unfortunately, their work does not immediately apply to our setting: whilst Lemma \ref{lem:no-invariant-discs} implies that, for suitable data, we have no Floer solutions lying in $X^{\inv} \subset X$ for topological reasons, there may be curves which lie in the fixed loci $X^{\iota_j}$ of either one of the two involutions.  In general the existence of such curves can obstruct the existence of any $\iota_j$-equivariant regular $J$, and hence \emph{a fortiori} of a $G$-equivariant regular $J$.  More precisely, if $x_+, x_-$ are $\iota_j$-invariant intersections of $S_a$ and $S_b$ for which the difference in their Maslov indices, computed in the fixed locus $X^{\iota_j}$, is greater than the difference of the indices computed in the total space, it is impossible to achieve equivariant transversality (because the dimension of the space of solutions in the fixed set is greater than the dimension in the total space). 

A sufficient criterion for vanishing of this obstruction was given in \cite[Section 3.5]{SS-localisation}. 
 
 \begin{definition} \label{defn:stable-normal}
 A {\em stable normal trivialization} consists of two pieces of data:
\begin{enumerate}
 \item  a stable trivialization $\phi:  \nu_{M^{\iota}}\times \bC^r \rightarrow M^{\iota} \times \bC^{n_{\anti}+r}$ of the normal bundle to the fixed point set, which is unitary (respects both the complex and symplectic structures); 
\item   homotopies $h_k: [0;1] \times L_k^{\inv} \rightarrow
Gr_{\mathrm{Lag}}(n_{\anti}+r)$, for $k=0,1$, such that (i) $h_0(0,\cdot) = \phi_*(\nu_{L_0^{\inv}} \times \bR^r)$
and  $h_0(1,\cdot) = \bR^{n_{\anti}+r}$, (ii) $h_1(0,\cdot) = \phi_*(\nu_{L_1^{\inv}} \times i\bR^r)$
and  $h_1(1,\cdot) = i\bR^{n_{\anti}+r}$. 
\end{enumerate}
\end{definition}

Let  $\hat{M}=M\times \bC^r$ with $\hat{\iota}=\iota\times (-1)$, then $\hat{M}^{\iota} = M^{\iota}$, and if we set $\hat{L}_0 = L_0\times \bR^r$ and $\hat{L}_1 = L_1\times i\bR^r$, then $\hat{L}_i^{\iota} = L_i^{\iota}$. We fix the convention that the linear Lagrangian subspaces $\bR$ and $i\bR$ in $\bC$ are graded so the rank one Floer cohomology $HF(\bR,i\bR)\cong \bZ$ is concentrated in degree zero.  There is a canonical isomorphism $\CF(L_0,L_1) \cong \CF(L_0\times \bR^r, L_1\times i\bR^r)$ of Floer complexes, where the first complex lives inside $M$ and the second inside $\hat{M}$ equipped with a product almost complex structure.

\begin{lemma} \label{lem:reduce-to-product}
Suppose $(M,L_0,L_1)$ admits a stable normal trivialization $\Upsilon = (\phi, h_0, h_1)$. There is a Hamiltonian symplectomorphism $\Phi=\Phi_{\Upsilon}: \hat{M}\rightarrow \hat{M}$ satisfying (i) $\Phi|_{M^{\iota}} = \id$; (ii) $\nu_{\Phi(\hat{L}_0)} = \bR^{n_{\anti}+r}$; (iii) $\nu_{\Phi(\hat{L}_1)} = i\bR^{n_{\anti}+r}$.
\end{lemma}

\begin{proof}
This is proved in the discussion after \cite[Definition 19]{SS-localisation}. 
\end{proof}

We have the Floer equation
\begin{equation} \label{eqn:Floer}
\begin{aligned}
 & \partial_s u + J_t(u) \big(\partial_t u - X_t(u)\big) = 0, \\
 & u(s,0) \in L_0, \;\; u(s,1) \in L_1.
\end{aligned}
\end{equation}
where $X_t$ denotes the time-dependent Hamiltonian flow of $H_t$. The next definition \ref{def:constrained-perturbation} and subsequent Lemma \ref{lem:generic-regular} are taken from \cite[Section 14c]{Seidel-FCPLT}. 

\begin{definition} \label{def:constrained-perturbation}
Suppose $(M,L_0,L_1)$ have stably trivial normal structure, and fix a stable normal trivialisation $\Upsilon$ and Hamiltonian symplectomorphism $\Phi=\Phi_{\Upsilon}$ as above.  We define \emph{$\Upsilon$-constrained Floer perturbation data}  for $(\Phi(\hat{L}_0), \Phi(\hat{L}_1))$  to comprise $(\{\hat{H}_t\}, \{\hat{J}_t\})_{0\leq t\leq 1}$, where 
\begin{enumerate} 
\item $\hat{H}_t$, $\hat{J}_t$ are $\hat{\iota}$-invariant; 
\item for some extension of the unitary trivialisation $\phi$ to a map $\bar\phi:\bC^r\times U(M^{\iota}) \rightarrow \bC^{n_{\anti}+r}$, the derivatives $\cdbar_{\hat{J}_t} \bar\phi$ vanish to first order along $M^{\iota}$; 
\item for some $h\in (0,\pi/2)$, the flow $X_t$ of $\hat{H}_t$ acts by rotation by $e^{iht}$ on $\nu_{L^{\iota}_0}$.
\end{enumerate}
\end{definition}

\begin{lemma} \label{lem:generic-regular}
$\Upsilon$-constrained Floer perturbation data are generically regular for all solutions $u:\bR\times[0,1] \rightarrow M^{\iota}$ to the equations \eqref{eqn:Floer} with image inside $M^{\iota}$.
\end{lemma}

\begin{proof}
We briefly review the discussion from \cite[Section 14c]{Seidel-FCPLT}.  Suppose $u$ satisfies \eqref{eqn:Floer} for $\Upsilon$-constrained perturbation data.   Using the first order information on $\hat{J}_t$ and $\hat{H}_t$ along $M^{\iota}$, \cite[Lemma 14.6]{Seidel-FCPLT} shows that, globally over $\bR\times[0,1]$, the anti-invariant part $D_{u,\hat{J}_t}^{\anti}$ of the linearisation can be written explicitly as an operator $\partial_s + i\partial_t + h$.  By a direct consideration of Fourier coefficients, \cite[Lemma 11.5]{Seidel-FCPLT} shows  that this is injective.  It follows that invariant curves $u$ are regular in the total space if and only if they are regular in the fixed point set.  To prove Lemma \ref{lem:generic-regular} it then suffices to show that the class of $\Upsilon$-constrained data is sufficiently large.  The $\hat{H}_t$ may be prescribed arbitrarily both on $M^{\iota}$ and outside a small open neighbourhood of $M^{\iota}$, provided that each is $\hat{\iota}$-invariant.  In a unitary splitting of $T\hat{M}|_{M^{\iota}}$, the off-diagonal parts  of the almost complex structures $\hat{J}_t$ vanish and the assumptions constrain the component tangential to $M^{\iota}$, but the normal part may also be prescribed arbitrarily.  This gives sufficient flexibility to appeal to standard transversality arguments to see that regular data are generic in this class.
\end{proof}

\begin{corollary}
Suppose that $(M,L_0,L_1)$ have stably trivial normal structure.  A choice of stable normal trivialisation $\Upsilon$ defines a (non-empty, open) set of $\hat{\iota}$-invariant time-dependent almost complex structures on $\hat{M}$ which are simultaneously regular for all moduli spaces of solutions to \eqref{eqn:Floer} in both $\hat{M}$ and $\hat{M}^{\iota} = M^{\iota}$.
\end{corollary}

\begin{proof}
For generic invariant time-dependent $J$, regularity holds at all curves $u: \bR\times[0,1] \rightarrow \hat{M}$ which are not entirely contained in $M^{\iota}$ (this is a direct adaptation of the usual regularity theory, and follows from \cite[Proposition 5.13]{KS}).    For curves $u$ inside the fixed locus, after stabilising and applying a Hamiltonian symplectomorphism, Lemma  \ref{lem:reduce-to-product} reduces us to a product situation, where we may invoke Lemma \ref{lem:generic-regular} to find an open set of $\Upsilon$-constrained perturbation data which are regular.
\end{proof}

The upshot of the discussion at this point is that, if a stable normal trivialisation exists, we can replace $(M,L_0,L_1, \iota)$ by $(\hat{M},\hat{L}_0, \hat{L}_1, \hat{\iota})$ in such a way that the Floer complexes $\CF(\hat{L}_0, \hat{L}_1) = \CF(L_0, L_1)$ are naturally identified (and similarly in the fixed point set $M^{\iota}$, which never changes); and the complex $\CF(\hat{L}_0, \hat{L}_1)$ admits a regular equivariant $J$. 

\subsection{Two involutions}

Now suppose that $M$ carries an action of $G = \bZ/2 \times \bZ/2$, generated by involutions $\iota_1$ and $\iota_2$, which again preserve the Lagrangians $L_j$ setwise.  We assume that the fixed submanifolds $M^{\iota_1}$ and $M^{\iota_2}$ meet cleanly in the global fixed locus $M^G$. 

\begin{lemma} \label{lem:iterate-localisation}
Suppose that 
\begin{enumerate}
\item $M$ admits an action of $\bZ/2 \times \bZ/2$ generated by $\iota_1$ and $\iota_2$;
\item the triple $(M, L_0, L_1)$ admits a stable normal polarisation for $\iota_1$;
\item the triple $(M^{\iota_1}, L_0^{\iota_1}, L_1^{\iota_1})$ admits a stable normal polarisation for $\iota_2|_{M^{\iota_1}}$;
\item the intersection points of $L_0$ and $L_1$ all lie in the common fixed locus $M^G$, so $L_0 \cap L_1 = L_0^G \cap L_1^G$.
\end{enumerate}
Then there is a cochain homotopy equivalence $\CF(L_0,L_1;\bZ/2) \simeq \CF(L_0^G, L_1^G;\bZ/2)$.
\end{lemma}

\begin{proof}
We apply the involutions iteratively. Existence of a stable normal trivialisation for $\iota_1$ gives an equivariant $J$ which identifies the Floer complexes $\CF(L_0,L_1; \bZ/2) = \CF(L_0^{\iota_1}, L_1^{\iota_1};\bZ/2)$, since all intersection points are fixed and all non-fixed discs come in pairs (cf. the proof of Lemma \ref{lem:equivariant-transversality-suffices}).  Note that, even though this identification of complexes arises from stabilisation (and goes via the intermediate complex $\CF(\hat{L}_0, \hat{L}_1)$), the fixed locus $M^{\iota_1}$ for the first involution is never changed. Since the involution $\iota_2|_{M^{\iota_1}}$ admits a stable normal trivialisation, we repeat the argument to identify $\CF(L_0^{\iota_1}, L_1^{\iota_1};\bZ/2)$ with $\CF(L_0^G, L_1^G;\bZ/2)$. 
\end{proof}

\begin{remark}
More generally, there is a natural notion of compatibility of stable normal trivialisations for a pair of involutions. The normal bundle satisfies $\nu_{M^G} \cong \nu_{M^{\iota_1}} \oplus \nu_{M^{\iota_2}}$, and  $\iota_2$ acts on the normal space $\nu_{M^{\iota_1}}$. This gives an obvious notion of a stable normal trivialisation for $\iota_1$ being $\iota_2$-equivariant. Since the ordering of the generators $\iota_j$ for $G$ was arbitrary, we will say that $(M, L_0, L_1)$ admits \emph{compatible} stable normal trivialisations.   By an equivariant version of Moser's theorem, the Hamiltonian symplectomorphism constructed in Lemma \ref{lem:reduce-to-product} can be constructed $\iota_2$-equivariantly (extending $\iota_2$ from $M$ to $\hat{M}$ by the trivial involution on the second factor). 
\end{remark}

\subsection{Constructing stable normal trivialisations}\label{sec:stable-normal-trivialisations}

We return to $X$,  its action of $G = \bZ/2 \times \bZ/2$, and the $G$-invariant Lagrangian matching spheres $S_a$ and $S_b$. The following is an analogue of \cite[Lemma 31]{SS-localisation}.

\begin{lemma}
The triple $(X, S_a,S_b)$ admits compatible stable normal trivialisations.
\end{lemma}

\begin{proof}
Consider the map $\bC^4 \times \bC^* \to \bC^2$ taking $(u_1,v_1,u_2,v_2,z) \mapsto (u_1v_1-z, u_2v_2-z)$. Then $X$ is the regular fibre over $(-1,1)$, so we have a short exact sequence of complex vector bundles
\begin{equation} \label{eqn:total_space_bundles}
0 \to TX \to T(\bC^4 \times \bC^*) \to T\bC^2 \to 0.
\end{equation}
There is a contractible space of splittings of such a sequence (from a Hermitian metric) which gives a canonical homotopy class of stable trivialisation of $TX$; thinking about the family of regular fibres $X_{c,d}$ over $\Conf_2(\bC^*) \subset \bC^2$ shows the canonical homotopy class of stable trivialisation is preserved by parallel transport and the $\PBr_3$-action. 

The Lagrangian $S_a$ arises as a vanishing cycle, which means we can find a disc $D \subset \bC^2$ normal to the diagonal and a one-parameter family $\scrX_D \to D$ with fibre $X \mapsto 1$ and such that $\scrX_D \supset \Delta_a$ contains a Lagrangian thimble fibred over an arc $[0,1] \subset D$ and with boundary $\Delta_a \cap X = S_a$.  

We have a commutative diagram of vector bundles over $S_a$,
\begin{equation} \label{eq:triv-triv}
\xymatrix{
0 \ar[r] & \ar[d]^{\cong} TS_a \otimes_\bR \bC \ar[r] & \ar[d]^{\cong} ((T\Delta_a)|_{S_a}) \otimes_\bR \bC \ar[r] & S_a \times \bC \ar[r] \ar[d]^{=} & 0 \\
0 \ar[r] & TX|_{S_a} \ar[r] & T(\scrX_D)|_{S_a}\ar[r] & S_a \times \bC \ar[r] & 0
}
\end{equation}
Since $\Delta_a$ is contractible, $T\Delta_a$ is trivial, and this induces a stable trivialization of $TS_a$; the nullhomotopy of the stabilised Lagrangian Gauss map required for Definition \ref{defn:stable-normal} thus comes from an extension of that map to $\Delta_a$.  The commutativity of \eqref{eq:triv-triv}  shows that the complexification of this  trivialization is (canonically) homotopic to the  trivialization of $TX|_{S_a}$ arising from the previous stable trivialisation of $TX$ from \eqref{eqn:total_space_bundles}. 

Because of our previous remarks about parallel transport, the same is true for any  matching sphere $S_b$, and furthermore the argument goes through equivariantly. More explicitly, $TX$ has a $G$-equivariant stable trivialization, compatible with a $G$-equivariant stable trivialisation of $TS_a$  in the sense that under the natural isomorphism $TS_a \otimes \bC \cong TX|_{S_a}$, there is an equivariant homotopy between these two trivializations. Restriction to fixed point sets of either generator of $G$ and to the corresponding anti-invariant directions along those fixed-point loci now yields the desired compatible stable normal trivializations.
\end{proof}

At this point we have satisfied all of the hypotheses of Lemma \ref{lem:iterate-localisation}, bearing in mind that in Lemma \ref{lem:intersection-points-invariant} we constructed Hamiltonian isotopies of the matching spheres after which $S_a \cap S_b = s_a \cap s_b$, i.e. all intersection points lie in the common fixed point locus. This completes the construction of the $J$ required for Lemma \ref{lem:equivariant-transversality-suffices}, and hence our discussion of the proof of Proposition \ref{prop:generalised_Khovanov_Seidel}.

 \subsection{Dynamics on surfaces}
 
 Having a formula for ranks of Floer cohomologies of matching spheres in terms of intersection numbers on the surface $X^{\inv}$ leads to growth rate bounds, using classical arguments from dynamics on surfaces.
 
\begin{proposition}\label{prop:growth}
Suppose that an autoequivalence $\phi \in \bZ^{\ast \infty}$ is such that for all $S_i, S_j \in \calS$, the rank of the Floer cohomology group $\HF(S_i, \phi^n S_j;\bZ/2)$ grows at most linearly with $n$. Then, up to shifts, $\phi$ is a power of a spherical twist in an element of $\calS$. 
\end{proposition}

\begin{proof}
The Nielsen-Thurston theorem (see \cite{FLP}) says that elements of mapping class groups of punctured surfaces are either periodic, reducible, or pseudo-Anosov.  Concretely for $\PBr_3$ these cases simplify as follows. The only periodic elements are powers of the full twist (i.e. elements in the centre of $\PBr_3$); the only reducible elements are powers of a full twist in a matching path between marked points; any other element is  pseudo-Anosov, which by definition means that it is pseudo-Anosov as a mapping of the four-punctured sphere given by collapsing and removing the boundary of the disc.  If $\phi$ is pseudo-Anosov, its lift $\hat\phi$ to the four-fold cover $X^{\inv}$ introduced in Proposition \ref{prop:generalised_Khovanov_Seidel}  is still pseudo-Anosov (as it preserves two transverse singular foliations, namely the preimages under the branched cover of the foliations preserved by $\phi$ downstairs). Then for any two simple closed curves $\ell, \ell'$ on $X^{\inv}$ the geometric intersection number $\iota(\ell, \hat\phi^k \ell')$ grows exponentially in $k$ \cite[Expos\'e 12, Th\'eor\`eme II]{FLP}. Taking $\ell, \ell'$ to be lifts of arcs $\gamma_i, \gamma_j$ downstairs, this geometric intersection number gives a lower bound for the rank of $\HF(S_i, \phi^k S_j;\bZ/2)$ by Proposition \ref{prop:generalised_Khovanov_Seidel}.

Suppose, then, that the growth of Floer cohomology is at most linear.  Since $\phi \in \bZ^{\ast\infty}$, if it is the image of a power of a full twist in a matching path, that matching path must be associated to some element of $\calS$. 
 If the growth rate is linear for some $S_i, S_j$, then $\phi$ is a shift of a non-trivial power of a Dehn twist in a matching sphere; if the growth rate is trivial for all $i,j$ then $\phi$ is the image of a central element, and acts by a shift. 
\end{proof}

\begin{lemma}\label{lem:fix-power-and-shift}
For a spherical twist $\phi = T_S$, the rank of $\HF(S_i, \phi^n S_j;\bZ/2)$ grows at most linearly with $n$.
\end{lemma}

\begin{proof} This follows immediately from the exact triangle for a spherical twist \eqref{eqn:twist-triangle}. \end{proof}

\subsection{A spherical twist is a power of a Dehn twist}
The conclusion of the previous section is that, as an autoequivalence of $D \cW(X)$,  $T_S = \tau_{\gamma}^l [k]$, where $\tau_\gamma$ denotes the Dehn twist in the matching sphere $S_\gamma$ associated to a matching path $\gamma$, and $k, l \in \bZ$ are to be determined.

\begin{lemma}
We have $k=0$, so the spherical twist $T_S = \tau_{\gamma}^l$ is the image of some power of a Dehn twist in a matching path.
\end{lemma}

\begin{proof}
By conjugating with a suitable element of $\bZ^{\ast \infty}$, we may assume without loss of generality that $S_\gamma = S_1$, mirror to $\calO_C(-1)$. In particular, since $S_1$ is disjoint from the Lagrangian thimble-like object $L_0$ (see Figure \ref{fig:spherical-conventions}), we have that 
\[
T_S(L_0) = \tau_{\gamma}^l(L_0)[k] = L_0[k].
\]
Suppose for contradiction that $k\neq 0$. Then in the exact triangle
\[
\xymatrix{
HF^\ast(S,L_0) \otimes S \ar[r] &  L_0 \ar[r] &  T_S(L_0) = L_0[k] \ar@/^1.0pc/[ll]
}
\]
the second arrow is given by an element of $HW^0(L_0,L_0[k]) = HW^k(L_0,L_0)$. But \cite[Remark 6.1]{CPU} computes that the endomorphisms of $L_0$ are concentrated in degree zero, so this vanishes.  This implies
\[
HF^\ast(S,L_0) \otimes S \simeq L_0 \oplus L_0[k]
\]
Taking self-endomorphisms of this object, since $HW^\ast(L_0,L_0) = SH^0(X)$ is infinite-dimensional, whilst $HF^\ast(S,S)$ has finite rank, we see that $HF^\ast(S,L_0)$ must be an infinite-dimensional vector space, which contradicts that $S$ is spherical.  We conclude that the shift  satisfies $k=0$.
\end{proof}

\subsection{Reduction to the `small category' $\scrC$}

Continuing with the same notation, in the course of the previous proof we showed that  $T_S(L_0) = L_0$. Strengthening this:

\begin{lemma}
$HF^\ast(S,L_0) = 0 = HF^\ast(L_0,S)$.
\end{lemma}

\begin{proof}
For any given $n \in \bZ$,

\[
HF^*(S,L_0) = HF^*(S, T_S^n(L_0)) = HF^*(T_S^{-n} S, L_0) = HF^*(S[-2n], L_0) = HF^{*+2n}(S,L_0)
\]
where we have used that for any 3-dimensional spherical object $\hat{S}$, the twist functor acts on the object itself by a non-trivial shift,  $T_{\hat{S}}(\hat{S}) = \hat{S}[2]$. 

Taking $n\neq 0$ shows that the vector space $HF^\ast(S,L_0)$ is periodic with some non-trivial grading shift (so infinite rank) or vanishes; again by the definition of $S$ being spherical, it must be the latter. Thus, $HF^\ast(S, L_0) = 0$.  By non-degeneracy of the pairing
\[
HF^\ast(S,L_0) \otimes HF^\ast(L_0,S) \to \bC
\]
(or just by repeating the argument with the order of the branes reversed)  we conclude that $HF^\ast(L_0,S) = 0$ also.
\end{proof}

In the course of the proof of Lemma \ref{lem:sphericals-act-trivially}, we showed that the mirror object $\calE_S \in D(Y)$ to $S$ belonged to the full subcategory $\scrD$ of complexes with cohomology supported on the curve $C$. There is a `small' category, cf. \cite{HW} for instance,
\[
\scrC = \left\{ F \in D(Y) \, | \, \supp(F) \subset C, \, R \pi_*(F) = 0\right\} \subset \scrD \subset D(Y)
\]
of objects of $D(Y)$ which are both cohomologically supported on $C$ and have trivial push-forward under the small resolution.  

\begin{lemma}
The mirror complex $\calE_S \in \scrD$ belongs to the subcategory $\scrC$.
\end{lemma}

\begin{proof}
Recall that $L_0$ is mirror to $\calO_Y$, so 
\[
HF^\ast(L_0,S) = 0 \Rightarrow \Ext^*(\calO_Y, \calE_S) = 0.
\]
Pushing down under the contraction $\pi: Y \to \Spec(R) = A$, we obtain
\[
\Ext^*(\calO_A, R\pi_* \calE_S) = 0.
\]
However, $A$ is affine and $\calO_A$ is a spanning object for the derived category $D(A)$, so this forces $R\pi_* \calE_S = 0$.  We deduce that $\calE_S \in \scrC$.
\end{proof}

\subsection{Conclusion}

To recap, if $S \in D\cW(X)$ is spherical, it has a mirror $\calE_S$ which is a spherical object in $\scrD\subset D(Y)$. 
We showed that the twist $T_{\calE_S}$ is mirror to $\tau_{S_\gamma}$ for some $S_\gamma \in \calS$. Conjugating with a $\PBr_3$ element taking $\gamma$ to $\gamma_1$, wlog $T_{\calE_S}$ is mirror to $\tau_{S_1}$, and we saw above that moreover we get that $\calE_S$ is a spherical object in the small category $\scrC \subset D(Y)$. 
The classification of spherical objects in  $\scrC$ is standard (cf. \cite{SW,HW}; in the terminology of \cite{SW},  in our case $\scrC$ corresponds to the one-vertex no-loop quiver). There is a unique spherical object up to shift, given by $\mathcal{O}_C(-1)$. Translating back under the mirror equivalence $D\scrW(X) \simeq D(Y)$, it follows that $S$ is quasi-isomorphic to some shift of $S_{\gamma}$.

This concludes the proof of Theorem \ref{thm:classification_sphericals}. Passing back through the mirror, this readily implies  Theorem \ref{thm:main2}.

\begin{remark} \label{rmk:coeff_fields2}
Using Remark \ref{rmk:coeff_fields}, and the fact that Lemma \ref{lem:fix-power-and-shift} works over any field, the previous analysis simultaneously classifies spherical objects in $D\scrW(X;\bZ)$ or $D\scrW(X;\bC)$. For instance, we deduce that $X$ contains no Lagrangian Lens space $L(p,q)$ with non-trivial fundamental group.
\end{remark}

\begin{remark}\label{rmk:entropy}
The variety $Y$ is defined over $\bZ$, and it seems likely that one could lift the equivalence $D\scrW(X) \simeq D(Y)$ of $\bC$-linear categories in Theorem \ref{thm:hms-equivalence} to an equivalence of categories over $\bZ$ (we have not checked the details). Assuming this, 
Propositions \ref{prop:generalised_Khovanov_Seidel} and \ref{prop:growth} would together imply that if $\phi \in D(Y)$ is a Torelli autoequivalence (i.e.~one acting trivially on $K$-theory) which is not a power of a spherical twist, and $S, S’$ are any spherical objects in $D(Y)$, then the total rank of $\Ext^*(\phi^k(S),S’;\bZ/2)$ grows exponentially in $k$ (moreover the exponential growth rate is an algebraic integer and is at least log(2)/12 \cite{Penner}, etc).  This follows from the identification of the $\bZ/2$-ranks of these morphism groups with geometric intersection numbers of curves on a surface, and classical results in surface dynamics. This connects to ideas around categorical entropy, which is expressed in  Theorem \cite[Theorem 2.6]{DHKK}  in terms of a growth rate of $\Ext$-groups for smooth proper categories (this does not directly apply here since the compact category $\scrF(X)$ is not smooth). A related entropy computation, for the special case $\phi = T_S\circ T_{S’}$, is given in  \cite{BK}. 
\end{remark}

\section{Miscellania}\label{sec:other}

\subsection{Addition of subcritical Weinstein handles}\label{sec:adding_handles}

We will need the following generalisation of Lemma \ref{lem:trivial_on_homology}. 

\begin{lemma}\label{lem:add-handles-fix-homology}
Suppose $X'$ is a Weinstein 6-manifold given by iteratively adding any sequence of Weinstein one- and two-handles to $X$. Let $f \in \Homeo_c (X')$ be a compactly supported homeomorphism. Then $f_\ast$ acts trivially on $H_3 (X'; \bZ)$. 
\end{lemma} 

\begin{proof} Take all homology groups to have coefficient ring $\bZ$.
Comparing with the proof of Lemma \ref{lem:trivial_on_homology}, it's enough to show that the map $H_3(X') \to H_3(X', \partial X')$ has vanishing image. Consider the non-degenerate intersection pairing
$$
H_3 (X') \times H_3 (X', \partial X') \to \bZ. 
$$
As $X \subset X'$ is given by adding Weinstein one- and two-handles, the inclusion induces an isomorphism $H_3 (X) \cong H_3(X')$. In particular, $H_3(X') \cong \bZ^2$, generated by $S_0$ and $S_1$. As the intersection pairing of $S_i$ and $S_j$ vanishes for any $i,j$, we see that for any $i$ the image of $S_i$ in $H_3 (X', \partial X')$ vanishes. This completes the proof. 
\end{proof}

\begin{corollary}\label{cor:infinite-sphere-orbits}
Suppose $X'$ is any Weinstein 6-manifold given by iteratively adding Weinstein one- and two-handles to $X$. Then under the action of $\pi_0 \Symp_c (X')$, there are infinitely many orbit sets of Lagrangian spheres in $X'$. 
\end{corollary}
This answers a folklore question often attributed to Fukaya.

\begin{proof}
The spheres $S_i$ belong to pairwise distinct classes in $H_3 (X'; \bZ)$, and any compactly supported symplectomorphism acts trivially on $H_3(X'; \bZ)$ by Lemma \ref{lem:add-handles-fix-homology}. 
\end{proof}

\begin{corollary}\label{cor:infinite-generation-with-handles}
Suppose $X'$ is a Weinstein 6-manifold given by iteratively adding any sequence of Weinstein one- and two-handles to $X$. Then we have a pair of group homomorphisms
$$
\bZ^{\ast \infty} \hookrightarrow \pi_0 \Symp_c (X') \twoheadrightarrow \bZ^{\ast \infty}
$$
which compose to the identity.
\end{corollary}

In particular, by adding a two-handle, we can find such $X'$s which are simply connected.

\begin{proof}
We have a natural isomorphism of wrapped Fukaya categories $D \cW(X) \simeq D \cW(X')$, compatible with an isomorphism between $SH^0 (X)$ and $SH^0(X')$.  
(By \cite{CDGG, GPS1}, the categories $\cW(X)$ and $\cW(X')$ are non-degenerate and equivalent. The equivalence induces an isomorphism of Hochschild (co)homologies. On the other hand, by \cite{Ganatra-thesis}, the non-degeneracy implies that $HH_*(\cW(X)) \to SH^*(X)$ is an isomorphism. It's then immediate that we get a compatible isomorphism between $SH^0 (X)$ and $SH^0(X')$.)

 This implies that  a compactly supported symplectomorphism $\phi \in \pi_0 \Symp_c (X')$ induces an $SH^0(X)$-linear autoequivalence of $D \cW(X)$. 
As before, combining mirror symmetry for $X$ with Lemmas \ref{lem:R-linear-fixes-points}, \ref{lem:R-linear-FM-support} and \ref{lem:mirror-to-point}, we get that $\phi$ preserves the subcategory $\langle S_0, S_1 \rangle$ (mirror to $\scrD$), and fixes the torus $(T, \zeta)$, as an object of $D \cW(X)$, up to a shift. 

By Lemma \ref{lem:add-handles-fix-homology}, $\phi$ acts as the identity $H_3(X'; \bZ) \cong H_3 (X; \bZ)$, i.e.~on the numerical $K$-theory $K_\text{num} \scrF(X'))$.
As before, Corollary \ref{cor:trivial-K-theory} implies that we get a map $\pi_0 \Symp_c (X') \to \bZ^{\ast \infty} \times 2\bZ$. Projecting gives a map $\pi_0 \Symp_c (X') \to \bZ^{\ast \infty}$, which is surjective  by considering the image of the Dehn twists. The claim then follows.
\end{proof}

\subsection{Non-exact deformations}\label{sec:non-exact-deformations}

By open-ness of the symplectic condition, for a small neighbourhood $U$ of $(0,0) \in H^2(X; \bR) \cong \bR^2$, there is a non-exact symplectic deformation of $(X, d\theta)$ to $(X, \omega_{a,b})$ with $[\omega_{a,b} ] = (a,b) \in U$.

We can in fact exhibit non-exact deformations $(X, \omega_{a,b})$ for arbitrary $(a,b) \in \bR^2$, via the Morse-Bott-Lefschetz fibration.
Let $t_c: T^\ast S^1 \to T^\ast S^1$ be a symplectomorphism of flux $c$ (translating the cylinder). 
Start with the previously considered Morse-Bott-Lefschetz fibration $\pi: X \to \bC^\ast$, with Morse-Bott singular fibres at $\pm 1$. 
 On a thickening of $\pi^{-1}(\bR_{<0})$, the fibrewise map $t_c \times \id$ is locally a symplectomorphism. This means that if we cut open $X$ along $\pi^{-1}(\bR_{<0})$ and glue back using $t_a \times \id$, then the resulting manifold, which is manifestly still diffeomorphic to $X$, inherits a symplectic form. This is no longer exact; for suitable coordinates on $H^2(X; \bR)$, it has class $(a,0)$. Similarly, we could cut open $X$ along $\pi^{-1}(\bR_{>0})$ and glue back using $\id \times t_b$, and obtain a symplectic form in class $(0,b)$. Combining the two operations, we get $(X, \omega_{a,b})$ for arbitrary pairs $(a,b) \in \bR^2$.

In the original Morse-Bott-Lefschetz fibration on $(X, \omega)$, for any vanishing path to $-1$, the tori $(t_c \times \id )\cdot \mu^{-1}(0)$ are all vanishing cycles. Similarly,  the tori $(\id \times t_c) \cdot \mu^{-1}(0)$ are all vanishing cycles for arbitrary vanishing paths to $1$. It follows that in case of the $(X, \omega_{a,b})$, for each matching path $\gamma_i$ from Figure  \ref{fig:spherical-conventions}, there is a choice of vanishing torus (depending on the winding number of $\gamma_i$ about 0) which gives a Lagrangian matching sphere. 

 By inspection, whenever $a$ and $b$ are both non-zero, $S_i$ and $S_j$ are disjoint if $i \neq j$. Examples for small $i$ are given in Figure \ref{fig:non-exact}.

\begin{figure}[htb]
\begin{center}
\includegraphics[scale=0.30]{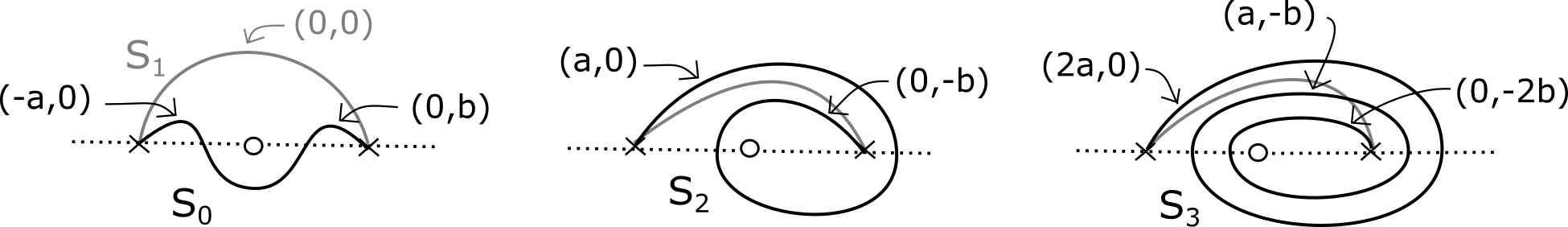}
\caption{Matching spheres in $(X, \omega_{a,b})$ and flux values for the vanishing tori (with respect to the upper half plane). $S_1$ is in grey in each diagram to visualise intersection points.}
\label{fig:non-exact}
\end{center}
\end{figure}

\begin{corollary}\label{cor:non-exact-twists}
 For $a \neq 0$ and $b \neq 0$,  there is an injection $\bZ^{\oplus \infty } \hookrightarrow \pi_0 \Symp_c (X, \omega_{a,b})$. 
\end{corollary}

\begin{proof}
The Dehn twists in the $S_i$ define a homomorphism $\bZ^{\oplus \infty } \to \pi_0 \Symp_c^{\gr} (X, \omega_{a,b})$. We claim this is injective by considering the action on graded Lagrangian spheres: each $S_i$ admits a $\bZ$-torsor of gradings, and the twist $\tau_{S_i}$ shifts the grading on $S_i$ by $2$; as they are disjoint, it does not change the grading on any $S_j$ for $i \neq j$. Finally, as the forgetful map from $\Symp_c^{\gr} (X, \omega_{a,b})$ to gradeable compactly supported symplectomorphisms of $(X, \omega_{a,b})$ splits, we get an injection 
$\bZ^{\oplus \infty } \hookrightarrow \pi_0 \Symp_c (X, \omega_{a,b})$ as required. 
\end{proof}

\begin{remark} By inspection, for our description of the $S_i$ as matching spheres, there is no compact subset of $X$ containing all of the matching spheres $S_i$: as $|i|$ inscreases, the $S_i$ have arbitrarily large maximal radial coordinate in the conical end for (the deformation of) $X$. 
\end{remark}

\subsection{Homeomorphism \& diffeomorphism groups}

\begin{lemma}
 The composite $\bZ^{\ast \infty} \to \pi_0 \Symp_c (X, d \theta) \to \pi_0 \Diff_c (X)$ factors through $\bZ^{\oplus \infty}$. 
\end{lemma}

\begin{proof}
This immediately follows by noting that this is the same map as the composite $\bZ^{\ast \infty} \to \pi_0 \Symp_c (X, \omega_{a,b}) \to \pi_0 \Diff_c (X)$ for any deformation $(X,\omega_{a,b})$ of $(X, d\theta)$, and recalling that the Dehn twists in the $S_i$ generate a $\bZ^{\oplus \infty}$ in $ \pi_0 \Symp_c (X, \omega_{a,b})$. 
\end{proof}

Recall that $\pi_1(X) = \bZ$. We view $X$ as the completion of a Weinstein domain $(X,\partial X)$ and let $(X', \partial X')$ denote the Weinstein domain obtained by adding one subcritical two-handle to $X$ so as to kill its fundamental group.

\begin{lemma} The boundary  $\partial X'$ is diffeomorphic to $(S^2\times S^3) \# (S^2\times S^3)$.
\end{lemma}

\begin{proof}
The boundary $\partial X$ has $\pi_1(\partial X) = \bZ$. On the boundary, the subcritical handle addition acts by a surgery, removing $S^1\times D^4$ and regluing $D^2\times S^3$, and so $\pi_1(\partial X') = \{1\}$. The long exact sequence for $(X,\partial X)$ shows that $H_2(\partial X) \cong H_2(X) \cong \bZ^2$, and hence $H_2(\partial X';\bZ) = \bZ^2$ also. In particular, the homology of $X'$ is torsion-free. Since $H^2(X';\bZ) \to H^2(X;\bZ)$ is an isomorphism, $X'$ still has vanishing first Chern class, which shows that $\partial X'$ is spin. Simply-connected spin 5-manifolds are classified \cite{Barden}; when the homology is torsion-free they are given by $(S^2\times S^3)^{\# r}$ where $r$ is the second Betti number. \end{proof}

Recall that we previously constructed a diagram 
\[
\xymatrix{
\bZ^{\ast\infty} \ar[r] \ar[d] & \pi_0\Symp_c(X') \ar[d] \\
\bZ^{\oplus \infty} \ar[r] & \pi_0\Diff_c(X')
}
\]

\begin{lemma}\label{lem:finite_rank}
The map $\bZ^{\oplus \infty} \to \pi_0\Diff_c(X')$ factors through a finite rank abelian group.
\end{lemma}

\begin{proof}
If $(M,\partial M)$ is a simply-connected six-manifold with simply-connected boundary, then Kupers shows \cite[Theorem A, Remark (iv)]{Kupers} that the mapping class group $\pi_0\Diff_c(M)$ has finite type in the sense that its classifying space admits a CW structure with only finitely many cells in each dimension.   The proof actually yields a stronger result, see \cite[Theorem 4.2]{Kupers}, which is that there is an exact sequence
\[
1 \to \Gamma \to \pi_0\Diff_c(M) \to G \to 1
\]
where $\Gamma$ is a finitely generated abelian group, and $G$ is an arithmetic group.  For any group fitting into such an exact sequence, any abelian subgroup is finitely generated. 
\end{proof}

\bibliography{bib}{}

\begin{thebibliography}{DHKK14}

\bibitem[AS21]{Auroux-Smith}
Denis Auroux and Ivan Smith.
\newblock Fukaya categories of surfaces, spherical objects and mapping class
  groups.
\newblock {\em Forum Math. Sigma}, 9:Paper No. e26, 50, 2021.

\bibitem[Aur07]{Auroux:Gokova}
Denis Auroux.
\newblock Mirror symmetry and {$T$}-duality in the complement of an
  anticanonical divisor.
\newblock {\em J. G\"{o}kova Geom. Topol. GGT}, 1:51--91, 2007.

\bibitem[Bal]{Ballard}
Matthew Ballard.
\newblock Sheaves on local {C}alabi-{Y}au varieties.
\newblock arXiv:0801.3499.

\bibitem[Bar65]{Barden}
Dennis Barden.
\newblock Simply connected five-manifolds.
\newblock {\em Ann. of Math. (2)}, 82:365--385, 1965.

\bibitem[BCZ17]{BCZ}
Arend Bayer, Alastair Craw, and Ziyu Zhang.
\newblock Nef divisors for moduli spaces of complexes with compact support.
\newblock {\em Selecta Math. (N.S.)}, 23(2):1507--1561, 2017.

\bibitem[BEE12]{BEE}
Fr\'{e}d\'{e}ric Bourgeois, Tobias Ekholm, and Yasha Eliashberg.
\newblock Effect of {L}egendrian surgery.
\newblock {\em Geom. Topol.}, 16(1):301--389, 2012.
\newblock With an appendix by Sheel Ganatra and Maksim Maydanskiy.

\bibitem[BK23]{BK}
Federico Barbacovi and Jongmyeong Kim.
\newblock Entropy of the composition of two spherical twists.
\newblock {\em Osaka J. Math.}, 60(3):653--670, 2023.

\bibitem[Bri07]{Bridgeland}
Tom Bridgeland.
\newblock Stability conditions on triangulated categories.
\newblock {\em Ann. of Math. (2)}, 166(2):317--345, 2007.

\bibitem[CDRGG]{CDGG}
Baptiste Chantraine, Georgios Dimitroglou~Rizell, Paolo Ghiggini, and Roman
  Golovko.
\newblock Geometric generation of the wrapped {F}ukaya category of {W}einstein
  manifolds and sectors.
\newblock {\em Ann. Sci. Ec. Norm. Sup\'er.}, 57(2024):1--85.
\newblock arXiv:1712.09126.

\bibitem[CPU16]{CPU}
Kwokwai Chan, Daniel Pomerleano, and Kazushi Ueda.
\newblock Lagrangian torus fibrations and homological mirror symmetry for the
  conifold.
\newblock {\em Comm. Math. Phys.}, 341(1):135--178, 2016.

\bibitem[CS14]{Canonaco-Stellari}
Alberto Canonaco and Paolo Stellari.
\newblock Fourier-{M}ukai functors in the supported case.
\newblock {\em Compos. Math.}, 150(8):1349--1383, 2014.

\bibitem[DHKK14]{DHKK}
George Dimitrov, Fabian Haiden, Ludmil Katzarkov, and Maxim Kontsevich.
\newblock Dynamical systems and categories.
\newblock In {\em The influence of {S}olomon {L}efschetz in geometry and
  topology}, volume 621 of {\em Contemp. Math.}, pages 133--170. Amer. Math.
  Soc., Providence, RI, 2014.

\bibitem[FLP79]{FLP}
Albert Fathi, Fran\c{c}ois Laudenbach, and Valentin Po\'enaru.
\newblock {\em Travaux de {T}hurston sur les surfaces}, volume~66 of {\em
  Ast\'{e}risque}.
\newblock Soci\'{e}t\'{e} Math\'{e}matique de France, Paris, 1979.
\newblock S\'{e}minaire Orsay, With an English summary.

\bibitem[Gan12]{Ganatra-thesis}
Sheel Ganatra.
\newblock {\em Symplectic {C}ohomology and {D}uality for the {W}rapped {F}ukaya
  {C}ategory}.
\newblock ProQuest LLC, Ann Arbor, MI, 2012.
\newblock Thesis (Ph.D.)--Massachusetts Institute of Technology. See
  arXiv:1304.7312.

\bibitem[GHK15]{GHK}
Mark Gross, Paul Hacking, and Sean Keel.
\newblock Mirror symmetry for log {C}alabi-{Y}au surfaces {I}.
\newblock {\em Publ. Math. Inst. Hautes \'{E}tudes Sci.}, 122:65--168, 2015.

\bibitem[GPS20]{GPS1}
Sheel Ganatra, John Pardon, and Vivek Shende.
\newblock Covariantly functorial wrapped {F}loer theory on {L}iouville sectors.
\newblock {\em Publ. Math. Inst. Hautes \'{E}tudes Sci.}, 131:73--200, 2020.

\bibitem[GPS24]{GPS2}
Sheel Ganatra, John Pardon, and Vivek Shende.
\newblock Sectorial descent for wrapped {F}ukaya categories.
\newblock {\em J. Amer. Math. Soc.}, 37(2):499--635, 2024.

\bibitem[HK]{HK2}
Paul Hacking and Ailsa Keating.
\newblock Symplectomorphisms of mirrors to log {C}alabi-{Y}au surfaces.
\newblock arXiv:2112.06797.

\bibitem[HK22]{HK}
Paul Hacking and Ailsa Keating.
\newblock Homological mirror symmetry for log {C}alabi-{Y}au surfaces.
\newblock {\em Geom. Topol.}, 26(8):3747--3833, 2022.
\newblock With an appendix by Wendelin Lutz.

\bibitem[Huy06]{Huybrechts}
D.~Huybrechts.
\newblock {\em Fourier-{M}ukai transforms in algebraic geometry}.
\newblock Oxford Mathematical Monographs. The Clarendon Press, Oxford
  University Press, Oxford, 2006.

\bibitem[HW]{HaraW}
Wahei Hara and Michael Wemyss.
\newblock Spherical objects in dimension two and three.
\newblock arXiv:2205.11552.

\bibitem[HW23]{HW}
Yuki Hirano and Michael Wemyss.
\newblock Stability conditions for 3-fold flops.
\newblock {\em Duke Math. J.}, 172(16):3105--3173, 2023.

\bibitem[IU05]{IU}
Akira Ishii and Hokuto Uehara.
\newblock Autoequivalences of derived categories on the minimal resolutions of
  {$A_n$}-singularities on surfaces.
\newblock {\em J. Differential Geom.}, 71(3):385--435, 2005.

\bibitem[Kea15]{Keating-tori}
Ailsa Keating.
\newblock Lagrangian tori in four-dimensional {M}ilnor fibres.
\newblock {\em Geom. Funct. Anal.}, 25(6):1822--1901, 2015.

\bibitem[KS02]{KS}
Mikhail Khovanov and Paul Seidel.
\newblock Quivers, {F}loer cohomology, and braid group actions.
\newblock {\em J. Amer. Math. Soc.}, 15(1):203--271, 2002.

\bibitem[Kup]{Kupers}
Alexander Kupers.
\newblock Mapping class groups of manifolds with boundary are of finite type.
\newblock arXiv:2204.01945.

\bibitem[LO10]{Lunts-Orlov}
Valery~A. Lunts and Dmitri~O. Orlov.
\newblock Uniqueness of enhancement for triangulated categories.
\newblock {\em J. Amer. Math. Soc.}, 23(3):853--908, 2010.

\bibitem[MM09]{Margalit-McCammond}
Dan Margalit and Jon McCammond.
\newblock Geometric presentations for the pure braid group.
\newblock {\em J. Knot Theory Ramifications}, 18(1):1--20, 2009.

\bibitem[Pen91]{Penner}
Robert Penner.
\newblock Bounds on least dilatations.
\newblock {\em Proc. Amer. Math. Soc.}, 113(2):443--450, 1991.

\bibitem[Pom]{Pomerleano_intrinsic}
Dan Pomerleano.
\newblock Intrinsic mirror symmetry and categorical crepant resolutions.
\newblock arXiv:2103.01200.

\bibitem[Rei83]{Reid}
Miles Reid.
\newblock Minimal models of canonical {$3$}-folds.
\newblock In {\em Algebraic varieties and analytic varieties ({T}okyo, 1981)},
  volume~1 of {\em Adv. Stud. Pure Math.}, pages 131--180. North-Holland,
  Amsterdam, 1983.

\bibitem[Sei]{Seidel:categorical_dynamics}
Paul Seidel.
\newblock Categorical dynamics and symplectic topology.
\newblock Lecture notes, available at https://math.mit.edu/~seidel.

\bibitem[Sei00]{Seidel_graded}
Paul Seidel.
\newblock Graded {L}agrangian submanifolds.
\newblock {\em Bull. Soc. Math. France}, 128(1):103--149, 2000.

\bibitem[Sei08a]{Seidel:biased}
Paul Seidel.
\newblock A biased view of symplectic cohomology.
\newblock In {\em Current developments in mathematics, 2006}, pages 211--253.
  Int. Press, Somerville, MA, 2008.

\bibitem[Sei08b]{Seidel-FCPLT}
Paul Seidel.
\newblock {\em Fukaya categories and {P}icard-{L}efschetz theory}.
\newblock Zurich Lectures in Advanced Mathematics. European Mathematical
  Society (EMS), Z\"{u}rich, 2008.

\bibitem[Sei12]{Seidel-Am-Milnor}
Paul Seidel.
\newblock Lagrangian homology spheres in {$(A_m)$} {M}ilnor fibres via {$\Bbb
  C^*$}-equivariant {$A_\infty$}-modules.
\newblock {\em Geom. Topol.}, 16(4):2343--2389, 2012.

\bibitem[Smi22]{Smirnov}
Gleb Smirnov.
\newblock Symplectic mapping class groups of {K}3 surfaces and
  {S}eiberg-{W}itten invariants.
\newblock {\em Geom. Funct. Anal.}, 32(2):280--301, 2022.

\bibitem[SS05]{Seidel-Smith}
Paul Seidel and Ivan Smith.
\newblock The symplectic topology of {R}amanujam's surface.
\newblock {\em Comment. Math. Helv.}, 80(4):859--881, 2005.

\bibitem[SS10]{SS-localisation}
Paul Seidel and Ivan Smith.
\newblock Localization for involutions in {F}loer cohomology.
\newblock {\em Geom. Funct. Anal.}, 20(6):1464--1501, 2010.

\bibitem[SS20]{Sheridan-Smith}
Nick Sheridan and Ivan Smith.
\newblock Symplectic topology of {$K3$} surfaces via mirror symmetry.
\newblock {\em J. Amer. Math. Soc.}, 33(3):875--915, 2020.

\bibitem[{Sta}18a]{stacks-project}
The {Stacks Project Authors}.
\newblock \textit{Stacks Project}.
\newblock \url{https://stacks.math.columbia.edu}, 2018.

\bibitem[Sta18b]{Starkston}
Laura Starkston.
\newblock Arboreal singularities in {W}einstein skeleta.
\newblock {\em Selecta Math. (N.S.)}, 24(5):4105--4140, 2018.

\bibitem[SW23]{SW}
Ivan Smith and Michael Wemyss.
\newblock Double bubble plumbings and two-curve flops.
\newblock {\em Selecta Math. (N.S.)}, 29(2):Paper No. 29, 62, 2023.

\bibitem[Tod08]{Toda1}
Yukinobu Toda.
\newblock Stability conditions and crepant small resolutions.
\newblock {\em Trans. Amer. Math. Soc.}, 360(11):6149--6178, 2008.

\bibitem[Tod09]{Toda2}
Yukinobu Toda.
\newblock Stability conditions and {C}alabi-{Y}au fibrations.
\newblock {\em J. Algebraic Geom.}, 18(1):101--133, 2009.

\bibitem[To{\"{e}}07]{Toen}
Bertrand To{\"{e}}n.
\newblock The homotopy theory of {$dg$}-categories and derived {M}orita theory.
\newblock {\em Invent. Math.}, 167(3):615--667, 2007.

\bibitem[Wu14]{WeiweiWu}
Weiwei Wu.
\newblock Exact {L}agrangians in {$A_n$}-surface singularities.
\newblock {\em Math. Ann.}, 359(1-2):153--168, 2014.

\end{thebibliography}
\bibliographystyle{alpha}

\includepdf[pages={1-5}]{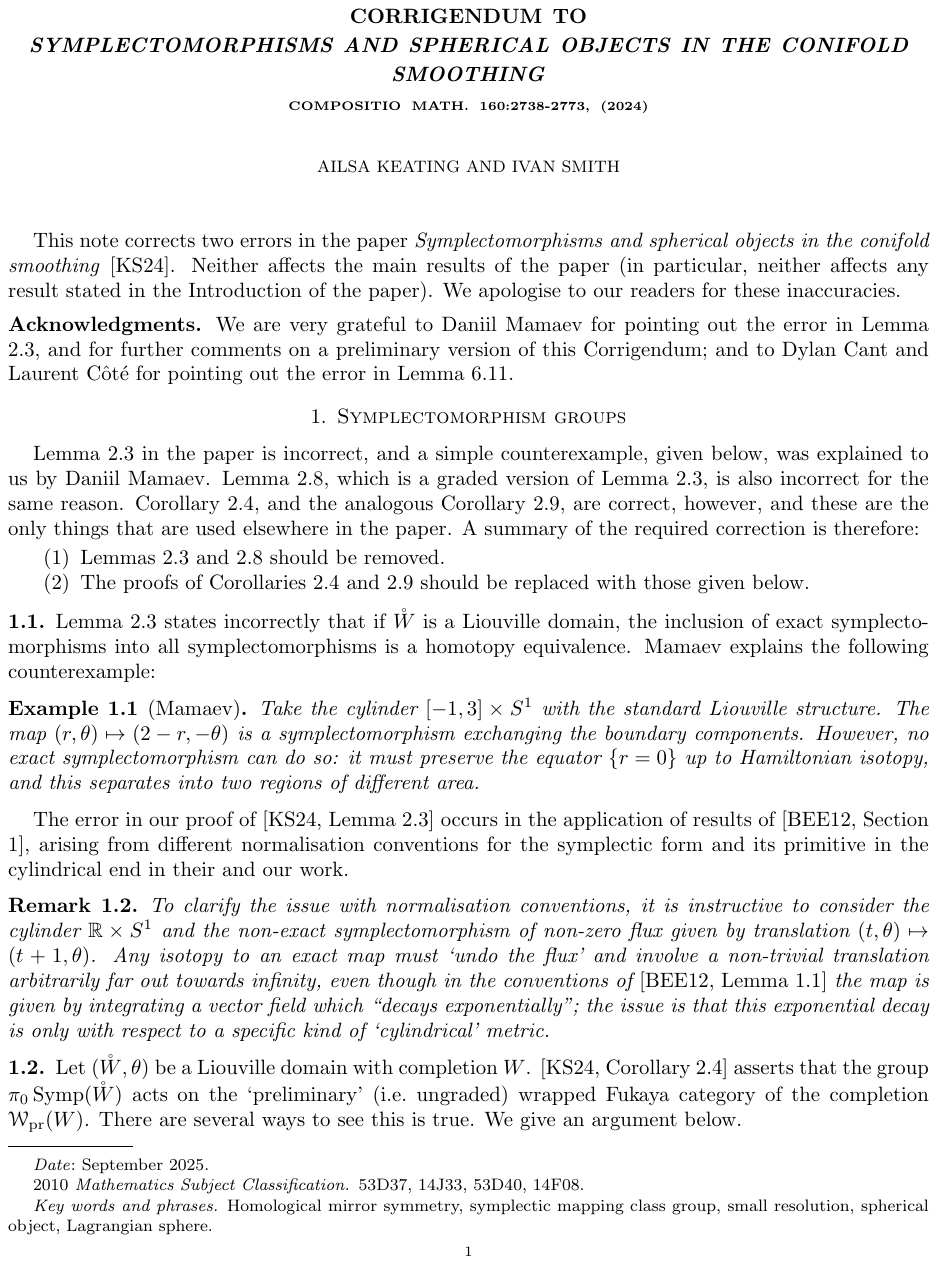}
\end{document}